\documentclass[12pt]{article}
\usepackage{latexsym, amsfonts, amscd, amssymb, verbatim, amsmath,
enumerate}

\newtheorem{theorem}{Theorem}[section]

\newtheorem{definition}{Definition}[section]

\newtheorem{lemma}{Lemma}[section]
\newtheorem{remark}{Remark}[section]

\newcommand{\proofbox}{\hspace{\fill}{$\Box$}}
\newenvironment{proof}{\textbf{Proof}.}{\proofbox}
\DeclareMathOperator{\cone}{cone}

\def\bar{\overline}
\date{}
\begin{document}

\title{ Locally H\"{o}lder continuity of the solution map to a boundary control problem with finite mixed control-state constraints}

\author{
	Nguyen Hai Son \footnote{%\textit{Corresponding author}; 
		School of Applied Mathematics and Informatics, Hanoi University of Science and Technology, No.1 Dai Co Viet, Hanoi, Vietnam; Email: son.nguyenhai1@hust.edu.vn} \ and Tuan Anh Dao \footnote{School of Applied Mathematics and Informatics, Hanoi University of Science and Technology, No.1 Dai Co Viet road, Hanoi, Vietnam; Email: anh.daotuan@hust.edu.vn}}
\maketitle

\noindent {\bf Abstract.} 
The local stability of the solution map to a parametric boundary control problem governed by semilinear elliptic equations with finite mixed pointwise constraints is considered in this paper. We prove that the solution map is locally H\"{o}lder continuous in $L^\infty-$norm of control variable when the strictly nonnegative second-order optimality conditions are satisfied for the unperturbed problem.

\medskip

\noindent {\bf  Key words.} Solution stability~$\cdot$~Locally   H\"{o}lder continuity~$\cdot$~Solution mapping~$\cdot$~Boundary control~$\cdot$~Mixed control-state constraints.

\medskip

\noindent {\bf AMS Subject Classifications.} 49K15~$\cdot$~90C29

\section{Introduction}

Let $\Omega$ be a bounded domain in $\mathbb{R}^2$ with the boundary
$\Gamma$ of class $C^{1,1}$. We consider the following
parametric boundary control problem. For each parameter $\lambda \in L^\infty(\Gamma)$, determine a control function $u \in L^2(\Gamma)$,
and a corresponding state function $y\in W^{1,r}(\Omega)$, $2<r<4$, which
\begin{align}
	\text{minimize}\  J(y, u, \lambda)=&\int_\Omega L(x,y(x))dx+ \notag \\
	&+\int_\Gamma\big(\ell(x,y(x), \lambda(x)) + \alpha(\lambda(x))u(x) + \frac{1}{2}\beta(\lambda(x))u^2(x)\big)ds, \label{P1} 
\end{align} subject to 
\begin{align}
	&\begin{cases}
		Ay  = 0  \quad &{\rm in}\ \Omega,\\
		\partial_{\nu} y + h(x,y, \lambda)= u +\lambda \quad  &{\rm on}\ \Gamma,
	\end{cases} \label{P2} \\
	& g_i(x,y(x),\lambda(x))+u(x) \leq 0\ {\rm a.e.}\ x\in\Gamma, \ i=1,2,...,m,
	\label{P3}
\end{align} where the maps $L, h:  \Omega\times  \mathbb R\to \mathbb R$, $\ell,\ g_i: \Gamma\times \mathbb R\times \mathbb R\to \mathbb R$
are  Carath\'{e}odory  functions fulfilling $L(x,\cdot),\ h(x,\cdot), \ \ell(x', \cdot, \cdot)$ and $g_i(x', \cdot, \cdot)$ with  $i=1,2,...,m$ are of class $C^2$ for a.e.
$x\in\Omega,\ x'\in\Gamma$; the functions $\alpha, \beta: \mathbb{R}\to\mathbb{R}$ are of class
$C^2$ and $m\geq 2$ is a positive integer number. Moreover, $A$ denotes a second-order elliptic operator of the form
$$ Ay(x)=-\sum_{i,j=1}^2 D_j(a_{ij}(x)D_iy(x))+a_0(x)y(x),  $$
where coefficients $a_{ij} \in C^{0,1}(\bar \Omega)$ satisfy $a_{ij}(x)=a_{ji}(x)$, $a_0 \in L^\infty(\Omega),\ a_0(x)\geq 0$ for a.e. $x\in \Omega$, $a_0 \not \equiv 0$, and there exists $C_0 >0$ such that
\begin{align}\label{C0}
	C_0 \Vert \xi \Vert^2 \leq \sum_{i,j=1}^2 a_{ij}\xi_i\xi_j \  \ \forall \xi=(\xi_1,\xi_2) \in \mathbb{R}^2 \quad \text{for a.e.} \quad x\in \Omega, \end{align}
%$\partial_{n_A}$ denote the conormal-derivative associated with $A$
$$ \partial_{\nu} y(x) = \sum_{i,j=1}^2 a_{ij}(x)D_iy(x)\nu_j(x) $$
in which $\nu(x)=(\nu_1(x),\nu_2(x))$ stands for the unit outward normal to $\Gamma$ at the point $x$.
Here, the measure on the boundary $\Gamma$ is the usual $1$-dimensional measure induced by the parametrization (see 
\cite{Troltzsch}). 

Throughout the paper, for each $\lambda \in L^\infty(\Gamma)$ let us denote by $(P(\lambda))$ the problem $\eqref{P1}-\eqref{P3}$,  by $\mathcal{F}(\lambda)$ the
feasible set of  $(P(\lambda))$ and by $\mathcal S(\lambda)$ the solution set of
$(P(\lambda))$.  Given a fixed parameter $\bar \lambda \in  L^\infty(\Gamma)$, we  call $(P(\bar \lambda))$
the unperturbed problem and assume that $\bar z=(\bar y, \bar
u)\in\mathcal{F}(\bar \lambda)$.  The goal of this paper is to study the behavior of $\mathcal S(\lambda)$ when $\lambda$ varies
around $\bar \lambda$ and  to estimate the error $\|\hat y_\lambda-\bar y\|+\|\hat u_\lambda-\bar u\| $ for $(\hat y_\lambda, \hat u_\lambda)\in \mathcal S(\lambda)$.

The solution stability for nonlinear optimal control problems plays a vital role in parameter estimation problems and in numerical
methods of finding optimal solutions (see \cite{Griesse1, Ito,Kien7}). The solution stability ensures that the solution set of perturbed problems is not very far away from the solution set of the unperturbed problems; in other words, the approximate solutions converge to the original solution.

The stability of the solution map to optimal control problems governed by elliptic equations has been studied by several authors recently. For some papers which have a close connection
to our problem,  we refer the readers to \cite{Alt},
\cite{Bonn}, \cite{Griesse}, \cite{Griesse1}, \cite{Kien1}, \cite{Kien5}, \cite{Malnowski1}, \cite{Nhu} and the references given therein. In \cite{Alt, Griesse}, the authors considered  problems, where the objective functionals are quadratic, the state equations and mixed pointwise constraints are linear. Hence the objective functionals are strongly convex and the feasible sets are convex. They showed  that under certain conditions, the solution maps are single-valued  and Lipschitz continuous with respect to parameters in $L^2$-norm (see \cite{Griesse}) and $L^\infty$-norm (see \cite{Alt}). When  the objective functional is not strongly convex or the feasible set is not convex then the solution map may not be singleton (see \cite{Kien1, Son1}). In this case, we have to use some tools of set-valued analysis and variational analysis to deal with these problems. In \cite{Kien1}, the lower semicontinuity of the solution map  with respect to parameters was showed for the distributed control problems. By using the direct method and the first-order necessary optimality conditions, Son al et. \cite{Son1, Son2} proved that the solution maps of the boundary control problems with one piontwise control-state constraint are upper semicontinuous and continuous in parametric.  Recently, Kien al et. \cite{Kien3}  used the no-gap second-order optimality conditions and the metric projections in Hilbert spaces to obtain the locally H\"{older} continuity of the solution map at the reference point not only in $L^2$-norm but also in $L^\infty$-norm of control variables. Notice that the authors in \cite{Kien3} considered the distributed control problem in which the number of constraints is two. 

In this paper, we are going to develop the method in \cite{Kien3} to obtain the same result for the boundary control problem \eqref{P1}--\eqref{P3}. Namely, if $(\bar y, \bar u)$, $(\hat y_\lambda, \hat u_\lambda)$ are locally optimal solutions to the problems $P(\bar \lambda)$ and $P(\lambda)$, respectively, then $\|\hat y_\lambda-\bar y\|_{W^{1,r}(\Omega)}+\|\hat u_\lambda-\bar u\|_{L^\infty(\Gamma)}\leq C\|\lambda-\bar \lambda\|_{L^\infty(\Gamma)}^{1/2}$ with a constant $C>0$ when $\lambda$ varies around $\bar \lambda$.  

We emphasize that the number of constraints in our problem is any positive integer $m$ and so, it may be greater than two. The estimation of  $\|\hat u_\lambda-\bar u\|_{L^\infty(\Gamma)}$, as we see in Section 3, depends on the number of Lagrange multipliers which equals the number of constraints. Therefore, the more constraints the problem has, the more difficult it is to derive some estimations. Consequently, although we follow the method in \cite{Kien3}, some significant improvement in the proof technique will be required. In particular, we have to establish the extra assumption (A5) which seems to generalize the assumption (H4) in \cite{Kien3}.    
%Therefore the estimation of  $\|\hat u_\lambda-\bar u\|_{L^\infty(\Gamma)}$  is more difficult than the problem which was considered in \cite{Kien3}.

The plan of this paper is as follows. In the next section, we set up  assumptions and  state our main result. Section 3 is devoted to the proof of the main result.

\section{Assumptions and the main result}

Hereafter given a Banach space $X$, $x_0\in X$ and $r>0$, let us denote by $X^*$ the dual space of $X$, and by  
$B_X(x_0, r)$, $\bar B_X(x_0, r)$ the open ball and the closed ball with center $x_0$ and radius $r$, respectively. We will write $B(x_0, r)$ and $\bar B(x_0, r)$ when no
confusion. Also, $B_X$ and $\bar B_X$ denote the open unit ball and the closed unit ball, respectively.

Let $\bar z=(\bar y, \bar u)\in W^{1,r}(\Omega)\times L^2(\Gamma)$. For a number $R>0$, we
define
\begin{align*}
	&\mathcal{F}_R(\lambda)=\mathcal{F}(\lambda)\cap B(\bar z, R),\\
	&\mathcal S_R(\lambda)=\big\{(y_\lambda, u_\lambda)\in  \mathcal{F}_R(\lambda)\mid J(y_\lambda, u_\lambda, \lambda)=
	\inf_{(y, u)\in \mathcal{F}_R(\lambda)}J(y,u,\lambda)\big\}.
\end{align*}
In this section, $\varphi:\Omega\times \mathbb{R}\to
\mathbb{R}$ is a function which stands for $L, h$,  and $\psi:\Gamma\times \mathbb{R}\times \mathbb{R}\to
\mathbb{R}$ is a function which stands for $\ell, g_i$ with $i=1,2,...,m$. Given a
couple $(\bar y, \bar u)\in \mathcal{F}(\bar \lambda)$, symbols $\varphi[x],  \varphi_y[x], \varphi[\cdot],  \varphi_y[\cdot]
$, $\psi[x],  \psi_y[x],
\psi[\cdot], \psi_y[\cdot]$. etc, stand for $\varphi(x,\bar y(x)),  \varphi_y(x,\bar y(x)), \varphi(\cdot,\bar y),  \varphi_y(\cdot,\bar y)
$, $ \psi(x, \bar y(x), \bar
\lambda(x))$, $\psi_y(x, \bar y(x),  \bar \lambda(x))$, $\psi(\cdot , \bar y(\cdot),
\bar \lambda(\cdot)),\ \psi_y(\cdot , \bar y(\cdot),
\bar \lambda(\cdot))$, etc, respectively.

We now impose   the following assumptions on $\varphi, \ \psi$, $\alpha$ and
$\beta$.

\begin{enumerate}
\item [(H1)] $\varphi:\Omega\times \mathbb R\to
\mathbb R$ is a Carath\'{e}odory function of class $C^2$ with
respect to the second variable and for each $M>0$, there exists a
positive number $k_{\varphi M}$ such that for a.e. $x\in\Omega$, one has
\begin{align*}
	|\varphi(x, y_1)- \varphi(x, y_2)| +|\varphi_y(x, y_1)- \varphi_y(x, y_2)| + |\varphi_{yy}(x, y_1)-&\varphi_{yy}(x,y_2)| \\
	& \leq k_{\varphi M}|y_1-y_2|
\end{align*} for all $ y_i$ satisfying $|y_i|\leq
M$ with $i= 1,2$. Furthermore, $\varphi(\cdot, 0)$ and $\varphi_y(\cdot, 0)$ belong to $L^\infty(\Omega)$.

\item [(H2)]  $\psi:\Gamma \times \mathbb R\times \mathbb R\to
\mathbb R$ is a Carath\'{e}odory function of class $C^2$ with
respect to the second variable and the third variables, and for each $M>0$, there exists a
positive number $k_{\psi M}$ such that for a.e. $x\in\Gamma$, one has
\begin{align*}
&|\psi(x, y_1, \lambda_1)- \psi(x, y_2, \lambda_2)| +|\psi_y(x, y_1, \lambda_1)- \psi_y(x, y_2, \lambda_2)| \\
&+ |\psi_{yy}(x, y_1,\lambda_1)-\psi_{yy}(x,y_2,\lambda_2)| +| \psi_{\lambda \lambda}(x, y_1, \lambda_1)- \psi_{\lambda \lambda}(x, y_2, \lambda_2)|\\
&+ |\psi_{y\lambda}(x,y_1,\lambda_1)- \psi_{y \lambda}(x, y_2,\lambda_2) |\\
& \leq k_{\psi M}(|y_1-y_2|+	|\lambda_1-\lambda_2|)
\end{align*} for all $ y_i,  \lambda_i \in \mathbb{R}$ satisfying $|y_i|, |\lambda_i|\leq
M$ with $i= 1,2$. Furthermore, $\psi(\cdot, 0, 0)$ and $\psi_y(\cdot, 0, 0)$ belong to $L^\infty(\Gamma)$.

\item [(H3)] The functions $\alpha$ and $\beta$ are of class
$C^2$ and for each $M>0$, there exist numbers
$k_{\alpha M}, k_{\beta M}>0$ such that
\begin{align*}
&|\alpha(\lambda_1)-\alpha(\lambda_2)|\leq k_\alpha |\lambda_1- \lambda_2|,\\
&|\beta(\lambda_1)-\beta(\lambda_2)| \leq k_\beta |\lambda_1- \lambda_2|
\end{align*}
for all $\lambda_i\in\mathbb{R}$ satisfying $|\lambda_i|\leq M$
with  $i=1,2$. Moreover, there exists $\gamma>0$ such that  $\beta(\bar \lambda(x))\geq \gamma$ for a.e. $x\in\Omega$.

\item[(H4)] $h(x, 0)=0$ and $h_y(x, y)\geq 0$, $g_{iy}[x]\geq 0$ for a.e. $x\in\Omega$,  for all
$y\in\mathbb{R}$, $i=1,2,...,m$. 

\item [(H5)] There exist measurable subsets $\Gamma_i$ of $\Gamma$, $i=1,2,...,m$ such that 
\begin{align*}
	&{\rm (i)}\ \Gamma= \bigcup_{i=1}^m\Gamma_i, \\
	&{\rm (ii)}\  \max_{ k\neq i} \operatorname*{ess~sup}\limits_{x\in\Gamma_i}(g_k[x]-g_i[x])< 0. 
\end{align*}
\end{enumerate}

It is noted that,  assumptions $(H1)-(H3)$ make sure that $J$, $h$ and $g_i$, $1\leq i\leq m$ are of class $C^2$ around the referent point, while assumption $(H4)$ guarantees existence and uniqueness of solution to the state equation \eqref{P2}, and that the Robinson constraint qualification condition is valid. Finally, as we will see in Section 3, assumption $(H5)$ is essential for estimations in $L^\infty$-norm of control variable.   

Let $s$ be the adjoint number of $r$, that is, $\frac{1}{r}+\frac{1}{s}=1$ and
$\tau: W^{1,s}(\Omega) \to W^{\frac{1}{r}, s}(\Gamma)$ be the trace operator. Recall that for each given $u \in W^{-\frac{1}{r},r}(\Gamma)$, a function $y \in W^{1,r}(\Omega)$
is said to be a (weak or variational) solution of \eqref{P2} if
\begin{multline*}
	\int_{\Omega}\sum_{i,j=1}^2 a_{ij}(x)D_iy(x)D_jv(x)dx+\int_{\Omega}a_0(x)y(x)v(x)dx+\int_{\Omega}h(x,y(x))v(x)dx\\
	=\int_{\Gamma}(u(x)+\lambda(x)) \tau v(x))ds
\end{multline*}
for all $v \in W^{1,s}(\Omega)$.  
It is known from \cite[ Theorem 3.1]{Son2} that  under the assumption $(H4)$,  for each $u\in L^2(\Gamma)$ and $\lambda \in
L^\infty(\Gamma)$, the equation \eqref{P2} has a unique solution $y\in W^{1,r}(\Omega)$ and there exists a
 constant $C_1>0$ such that
\begin{align}
	\|y\|_{W^{1,r}(\Omega)}\leq C_1(\|u\|_{L^2(\Gamma)} +\|\lambda\|_{L^2(\Gamma)}). \label{state}
\end{align} 

From $(H1)-(H3)$, it follows that for each fixed parameter  $\lambda \in B_{L^\infty(\Gamma)} (\bar \lambda,\epsilon_1)$ and for any $(\hat y, \hat u)\in \mathcal{F}(\lambda)$ and $(y,u)\in W^{1,r}(\Omega)\times L^2(\Gamma)$, we have
\begin{align*}
\langle \nabla_z J(\hat y, \hat u,\lambda), (y,u)\rangle =&\int_\Omega L_y(x,\hat y(x))y(x)dx+ \\ &\quad  \int_\Gamma\big(\ell_y(x, \hat y(x), \lambda(x))y(x)+ \alpha(\lambda(x))u(x)+\beta(\lambda(x))\hat u(x) u(x) \big)ds
\end{align*}
and
\begin{align*}
 \langle \nabla_z J(\bar y, \bar u, \bar \lambda), (y,u)\rangle=&\int_\Omega L_y(x,\bar y(x))y(x)dx+ \\ &\quad \int_\Gamma\big(\ell_y(x, \bar y(x),\bar \lambda(x))y(x)+ \alpha(\bar \lambda(x))u(x)+\beta(\bar \lambda(x))\bar u(x) u(x) \big)ds.
\end{align*}

\begin{definition}\label{DefCrD}
 	We say that a couple $(y, u)\in W^{1,r}(\Omega) \times L^2(\Gamma)$ is a critical direction for $(P(\bar \lambda))$ at $(\bar y,\bar  u)$ if the following conditions are fulfilled:

	\noindent $(i)$ $ \nabla_z J(\bar y,\bar  u, \bar \lambda), (y,u)\leq 0$,

	\noindent $(ii)$ $\begin{cases}
		A y + h_y[\cdot]y =0 &\text{ in } \Omega,\\
		\partial_\nu y=u &\text{ on } \Gamma, \end{cases} $ 
	
	\noindent $(iii)$ $g_{iy}[x]y(x) +u(x)\leq 0$
	whenever  $g_i[x]+ \bar u(x)=0$,\ $i=1,2,...,m$. 
\end{definition}
Denote by $\mathcal{C}[(\bar y,
\bar u), \bar \lambda]$ the set of critical directions of $(P(\bar \lambda))$ at
$\bar{z}=(\bar y,\bar  u)$. It is clear that $\mathcal{C}[(\bar y, \bar u), \bar \lambda]$ is a closed and convex cone.

\begin{definition}\label{LagrangeMul}
We say that the functions $\vartheta_\lambda \in W^{1,s}(\Omega)$ and  $e_{ \lambda  1}, e_{\lambda  2},..., e_{ \lambda  m} \in L^2(\Gamma)$ are Lagrange multipliers of $(P(\lambda))$ at $(\hat  y_\lambda, \hat  u_\lambda)$  if they satisfy the following conditions:

\noindent $(i)$ The  adjoint equation:
\begin{align}\label{adj}\begin{cases}
	-A \vartheta_\lambda +h_y(\cdot, \hat  y_\lambda)\vartheta_\lambda=-L_y(\cdot,\hat  y_\lambda) &\text{ in } \Omega \\
	\partial_\nu \vartheta_\lambda =-\ell_y(\cdot,\hat  y_\lambda)-\sum_{i=1}^{m}g_{iy}(\cdot,\hat y_\lambda, \lambda)e_{\lambda i}  &\text{ on } \Gamma 
\end{cases} \end{align} 

\noindent $(ii)$ The stationary condition in $u$:
\begin{align}\label{sta}			
-\vartheta_\lambda +\alpha(\lambda)+\beta(\lambda)\hat  u_\lambda +\sum_{i=1}^{m}e_{\lambda i}=0;
\end{align}

\noindent $(iii)$ The complementary conditions: 
	\begin{align}\label{com}	
     e_{\lambda i}(x)\geq 0, \quad e_{\lambda i}(x)(g_i(x,\hat  y_\lambda(x), \lambda(x))+ \hat  u_\lambda(x))=0
	\end{align} for each $i\in \{1,2,...,m\}$ and for a.e. $x\in\Omega$.  
\end{definition}
 Let us denote by $\Lambda[(\hat  y_\lambda, \hat  u_\lambda),  \lambda]$ and $\Lambda[(\bar y, \bar u),  \bar \lambda]$
the sets of Lagrange multipliers of $(P(\lambda))$ at $(\hat  y_\lambda, \hat  u_\lambda)$   and $(P(\bar \lambda))$ at $(\bar y, \bar u)$, respectively. 

\begin{definition}
We say that $\bar z=(\bar y, \bar u)$ is a locally strongly optimal solution of $(P(\bar \lambda))$ if  there exist numbers $\delta>0$ and $\sigma>0$ such that 
\begin{align*}%\label{LocallyStrongSol}
	J(z, \bar \lambda)\geq J(\bar z, \bar \lambda) +\sigma\|z-\bar z\|^2\quad \forall z\in B(\bar z, \delta)\cap\mathcal{F}(\bar \lambda). 
\end{align*}	
	\end{definition}
   We now state the main result of the paper.

\begin{theorem} \label{Theoremkey}
Suppose that the assumptions $(H1)-(H4)$ are fulfilled and $(\vartheta, e_1, e_2,..., e_m)\in \Lambda[(\bar y, \bar u) , \bar \lambda]$ such that
\begin{align} \label{StrictSOC}
\int_\Omega & \big(L_{yy}[x]y^2(x)+\vartheta h_{yy}[x]y^2(x) \big)dx  \notag \\ &+ \int_\Gamma\Big( \ell_{yy}[x]y^2(x)+  \beta (\bar{\lambda}(x))u^2(x) + 
\sum_{i=1}^{m}e_i(x) g_{iyy}[x]y^2(x) \Big) ds >0
\end{align} for all $(y, u)\in \mathcal{C}[(\bar y, \bar u), \bar \lambda]\setminus\{(0,0)\}$. Then $(\bar y,\bar u)$ is a locally strongly optimal solution of $(P(\bar \lambda))$ and there exist positive numbers $R_0, s_0$ and $l_0$ such that for all $\lambda \in B_{L^{\infty}(\Gamma)}(\bar \lambda, s_0)$,   every couple $(\hat y_\lambda, \hat u_\lambda)\in \mathcal S_{R_0}(\lambda)$ is a locally optimal solution of $(P(\lambda))$,   $\hat u_\lambda \in L^\infty(\Gamma)$  and
\begin{align}\label{HolderCond1}
\|\hat y_\lambda-\bar y\|_{W^{1,r}(\Omega)} + \|\hat u_\lambda-\bar u\|_{{L^2(\Gamma)}} \leq  l_0\|\lambda-\bar \lambda\|_{L^\infty(\Gamma)}^{1/2}.
\end{align} In addition, if $(H5)$ is satisfied then  there exist positive numbers $s_1$ and $l_1$ such that
for all $\lambda \in B_{L^{\infty}(\Gamma)}(\bar \lambda, s_1)$ and   $(\hat y_\lambda, \hat u_\lambda)\in \mathcal S_{R_0}(\lambda)$ one has
\begin{align}\label{HolderCond2}
	\|\hat y_\lambda-\bar y\|_{W^{1,r}(\Omega)}  + \|\hat u_\lambda-\bar u\|_{L^\infty(\Gamma)} \leq  l_1\|\lambda-\bar \lambda\|_{L^\infty(\Gamma)}^{1/2}.
\end{align}
\end{theorem}

\section{Proof of the main result}

In this section, we present the proof of Theorem \ref{Theoremkey}.  The proof will be divided into a sequence of lemmas. 

We begin by proving the solution stability in $L^2$-norm of control variable. To this end, we shall transfer problem $P(\lambda)$ to a mathematical programming problem and apply \cite[Proposition 3.1]{Kien3}. For this, we put
\begin{align*}
	& Y:=W^{1,r}(\Omega), \ U:=L^2(\Gamma),\ Z:=Y \times U, \Pi=L^\infty(\Gamma),\\
	& E_0:=(W^{1,s}(\Omega))^*,\ E:=(L^2(\Gamma))^m, K=\prod_{i=1}^{m}K_i 
\end{align*}  where
\begin{align*}
K_i=\{v\in L^2(\Gamma) \mid v(x)\leq 0 \quad  \text{a.e. } x\in \Gamma \}, \quad i=1,2,...,m.
\end{align*}

These spaces are separable and reflexive Banach space and $K$ is a unbounded, convex and closed subset in $E$. Let us define mappings 
\begin{align*}
 &H:  Y\times U\times \Pi\to E_0,\  H(y,u,\lambda)=  \sum_{i,j=1}^{2}a_{ij}D_iyD_j +a_0y +h(\cdot,y) +\tau^*(u+\lambda),\\
& G: Y\times U\times \Pi\to E,\  G(y,u, \lambda)=(G_1(y, u, \lambda), G_2(y, u, \lambda),...,G_m(y, u, \lambda)),
\end{align*} where  $\tau^* $  is the adjoint operator of the trace operator $ \tau $,  and 
$$\langle H(y; u; \lambda); v \rangle := \int_\Omega \big(\sum_{i,j=1}^{2}a_{ij}D_iyD_jv +a_0yv +h(x,y)v\big)dx +\int_\Gamma(u+\lambda)\tau vds \ \forall  v\in W^{1,s}(\Omega), $$\\ 
and $G_i(y, u, \lambda):=g_i(\cdot,y,\lambda)+u$, $i=1,2,...,m$.
 Also, we define the set 
 \begin{equation*}
 	D(\lambda)=\{z=(y,u)\in Y\times U | H(y,u,\lambda)=0\}.\end{equation*} 

Since $r>2$, $Y\hookrightarrow C(\bar \Omega)$. Moreover, analysis similar to that in \cite{Son} shows that the mappings $F$ and $G$ are well-defined.

Therefore, the  problem $(P(\lambda))$ can be written in the following form 
\begin{equation}\notag
	(P(\lambda))\quad\quad\quad
	\begin{cases}
		& J(y,u,\lambda)\to \min,\\
		& H(y, u, \lambda)=0, \\
		& G(y, u, \lambda)\in K,\\
	\end{cases}
\end{equation} or in the simpler form
\begin{equation}\notag
	(P(\lambda))\quad\quad\quad
	\begin{cases}
		& J(z, \lambda)\to \min,\\
		&z\in\mathcal{F}(\lambda),
	\end{cases}
\end{equation}
where $\mathcal{F}(\lambda):=D(\lambda)\cap G^{-1}(K)$ is the feasible set of problem $P(\lambda)$.

Recall that for given a closed set $D$ in $Z$ and a point $z\in D$, {\it the adjacent tangent cone} and
{\it the contingent cone} to $D$ at $z$ are defined by 
\begin{align*}
	T^{\flat}(D, z)&=\{h\in Z|\forall t_n\to 0^+, \exists h_n\to h, z+ t_nh_n\in D\},\\
	T(D, z)&=\{h\in Z|\exists t_n\to 0^+, \exists h_n\to h, z+ t_nh_n\in D\},
\end{align*}respectively. These cones are closed and $T^{\flat}(D, z)\subseteq T(D, z).$
It is well known that when $D$ is convex,  then
$$
T^{\flat}(D, z)= T(D, z)=\overline{{\rm cone}}(D-z)
$$ and the normal cone to $D$ at $z$ is given by
$$
N(D, z)=\{z^*\in Z^*\mid \langle z^*, d-z\rangle\leq 0\ \forall d\in
D\},
$$ where ${\rm cone}(D-z):=\{k(d-z): d\in D, k>0\}$ is the cone generated by $(D-z)$. 

\begin{lemma} \label{key1}
	If the assumptions $(H1)$-$(H4)$ are fulfilled then the following conditions hold: 
	\begin{enumerate}
		\item [$(i)$] There exist positive numbers $r_1, r_1', r_1''$
		such that for any $\lambda \in B_\Pi(\bar \lambda, r_1'')$, the mapping $J(\cdot,\cdot, \lambda)$, $H(\cdot, \cdot, \lambda)$ and $G(\cdot, \cdot, \lambda)$ are twice Fr\'{e}chet differentiable
		on $B_Y(\bar{y}, r_1)\times B_U(\bar{u}, r_1')$. The mapping $H(\cdot, \cdot, \cdot)$ is continuously Fr\'{e}chet differentiable on
		$B_Y(\bar y, r_1)\times B_U( \bar u, r_1')\times B_\Pi(\bar \lambda, r_1'')$.
		
		\item [$(ii)$] There exist constants $k_G>0$ and $k_J>0$ such that
		\begin{align*}
			&\|G(z_1, \lambda_1)-G(z_2,  \lambda_2)\|+\|\nabla G(z_1, \lambda_1)-\nabla G(z_2,  \lambda_2)\|\leq k_G(\|z_1-z_2\|+ \| \lambda_1- \lambda_2\|),\\
			&\|J(z_1, \lambda_1)-J(z_2, \lambda_2)\|+\|\nabla J(z_1, \lambda_1)-\nabla J(z_2, \lambda_2)\|\leq k_J(\|z_1-z_2\|+ \|\lambda_1-\lambda_2\|)
		\end{align*} for all $z_1, z_2\in B_Y(\bar y, r_1)\times B_U(\bar u, r_1')$ and $\lambda_1, \lambda_2\in B_\Pi(\bar \lambda, r_1'').$
		
		\item [$(iii)$] The mapping $H_y(\bar z,\bar  \lambda)$ is  bijective.

		\item [$(iv)$] $E= \nabla_z G(\bar z, \bar \lambda)(T(D(\bar \lambda), \bar z)) -\cone (K-G(\bar z, \bar \lambda)).
		$
		
	\end{enumerate}	
 \end{lemma}
\begin{remark} The condition $(iv)$ in Lemma \ref{key1} means that the unperturbed problem $P(\bar \lambda)$ satisfies the Robinson constraint qualification. According to \cite[Theorem 2.5]{Kien0}, this condition is equivalent to that there exists a number $\rho_0>0$ such that 
$$	B_E(0, \rho_0)\subset \nabla_z G( \bar z, \bar \lambda)(T(D(\bar \lambda), \bar z)\cap B_Z)-(K-G(\bar z, \bar \lambda))\cap B_E.
$$ \end{remark}	
\begin{proof}
Obviously, the assumptions $(H1)-(H3)$ imply the conditions $(i)$ and $(ii)$. Moreover, for  all $\lambda \in B_Z (\bar \lambda,
r_1'')$ and $(\hat y, \hat u)\in \mathcal{F}(\lambda)$, an easy computation shows that
\begin{align*}
&\langle \nabla_z J(\hat y, \hat u, w), (y,u)\rangle=\int_\Omega L_y(x, \hat y(x))y(x)dx \\
&\qquad  \qquad \qquad \qquad +\int_\Gamma \big(\ell_y(x, \hat y(x), \lambda(x))y(x)+ \alpha(\lambda(x))u(x)+\beta(\lambda(x))\hat u(x) u(x) \big)ds\\
&\langle \nabla_z J(\bar y,\bar  u,\bar  w), (y,u)\rangle=\int_\Omega L_y(x, \bar y(x))y(x)dx \\
&\qquad  \qquad \qquad \qquad +\int_\Gamma\big(\ell_y(x,\bar  y(x),\bar  \lambda(x))y(x)+ \alpha(\bar \lambda(x))u(x)+ \beta(\bar \lambda(x))\bar u(x) u(x) \big)dx,\\
&\nabla_y H(\bar z, \bar \lambda)y=\sum_{i,j=1}^{2}a_{ij}D_i y D_j +a_0 y +h_y[\cdot]y, \\ 
&\nabla_u H(\bar z, \bar \lambda)u=-\tau^*u,\\
&\nabla_z G(\bar z, \bar \lambda)=\big (\nabla_z G_1(\bar z, \bar \lambda), \nabla_z G_2(\bar z, \bar \lambda),...,\nabla_z G_m(\bar z, \bar \lambda)\big ),\\
& \nabla_z G_i(\bar z, \bar \lambda)=(g_{iy}[\cdot],  I_U) \ \forall i=1,2,...,m,  
\end{align*} where $I_U$ is the identity mapping of $U$. 

	In order to show $(iii)$, we take any $ f_0\in E_0$ and consider the equation
$$
H_y(\bar y, \bar u, \bar \lambda)y=f_0
$$ which is equivalent to 
\begin{align}\begin{cases}
	A y + h_y[\cdot]y =f_0 &\text{ in } \Omega\\
	\partial_\nu y=0 &\text{ on } \Gamma. \end{cases} \label{eq3.1}
\end{align} By $(H4)$, we have $h_y[x]\geq 0$ for a.e. $x\in\Omega$.  Therefore, the equation \eqref{eq3.1} has a unique solution $y\in W^{1,r}(\Omega)$ and so $(iii)$ is valid.

It remains to prove $(iv)$. It follows from $(iii)$ that  $\nabla H(\bar z, \bar \lambda)$ is surjective.  By \cite[Lemma 2.2]{Kien0},  we have
\begin{align*}
	T(D(\bar \lambda), \bar z)&=\big\{(y,u)\in Z \big| H_y(\bar z, \bar \lambda)y+H_u(\bar z, \bar \lambda)u=0\big\}\\
	 &=\left\{(y,u)\in Z \Big|\begin{cases}
		                           	A y + h_y[\cdot]y =0 &\text{ in } \Omega\\
		                           	\partial_\nu y=u &\text{ on } \Gamma \end{cases} \right\}.
\end{align*}
Therefore $(iv)$ is equivalent to saying that, for any
$(f_1, f_2,...,f_m)\in (L^2(\Gamma))^m$, there exist $(y,
u)\in Y\times U$ and $(v_1, v_2,...,v_m)\in {\rm cone}(K-G(\bar z, \bar \lambda))$
such that
\begin{align}\label{H3Assump}
&\begin{cases}
	A y + h_y[\cdot]y =0 &\text{ in } \Omega\\
	\partial_\nu y=u &\text{ on } \Gamma \end{cases} \notag \\
&f_i=g_{iy}[\cdot]y+  u-v_i,\quad i=1,2,...,m. 
\end{align}
Write $f_i= f_{i1}-f_{i2}$, where
$f_{i1}, f_{i2}\in K_i$, $i=1,2,...,m$. Set 
 $u_i= f_{i1}$ with $i=1,2,...,m$ and $u=\sum_{i=1}^{m}u_i$.
Since $f_{i1}(x)\leq 0$ for all $i=1,2,...,m$, one has $u(x)\leq 0$ a.e. $x\in \Gamma$.  

Now, we consider the following equation 
\begin{align} \label{themeq}
	\begin{cases}
		A y + h_y[\cdot]y =0 &\text{ in } \Omega\\
		\partial_\nu y=u &\text{ on } \Gamma .\end{cases}
\end{align}
By the assumption (H4), the equation \eqref{themeq} has a unique
solution $y_u\in Y$, that is 
\begin{align}\label{themeq2}
\int_{\Omega}\sum_{i,j=1}^2 a_{ij}(x)D_iy_u(x)D_jv(x)dx+\int_{\Omega}a_0(x)y_u(x)v(x)dx+\int_{\Omega}&h_y[x]y_u(x)v(x)dx\notag\\
&=\int_{\Gamma}u(x) \tau v(x)ds	
\end{align}  
for all $v \in W^{1,s}(\Omega)$. Setting $y_u^+=\max\{0,y_u\}$,  we have $y_u^+ \geq 0$ for a.e. $x\in \Gamma$ and $y_u^+ \in W^{1,2}(\Omega) \hookrightarrow  W^{1,s}(\Omega)$. Taking $\varphi =y^+_u$ in the equation \eqref{themeq2}, we obtain 
\begin{align*}
	\int_{\Omega}\sum_{i,j=1}^2 a_{ij}(x)D_iy_u(x)D_jy^+_u(x)dx+\int_{\Omega}a_0(x)y_u(x)y^+_u(x)dx+\int_{\Omega}&h_y[x]y_u(x)y^+_u(x)dx\notag\\
	&=\int_{\Gamma}u(x) \tau y^+_u(x)ds,	
\end{align*} 
which is equivalent to 
\begin{align*}
	\int_{\Omega}\sum_{i,j=1}^2 a_{ij}(x)D_iy^+_u(x)D_jy^+_u(x)dx+\int_{\Omega}a_0(x)y^+_u(x)y^+_u(x)dx+\int_{\Omega}&h_y[x]y^+_u(x)y^+_u(x)dx\notag\\
	&=\int_{\Gamma}u(x) \tau y^+_u(x)ds.	
\end{align*}
 Combining this with \eqref{C0} and the assumption $(H4)$ yields 
$$C'_0 \|y_u^+\|^2_{H^1(\Omega)}\leq \int_{\Gamma}u(x) \tau y_u^+(x)ds \leq 0$$ for a constant $C'_0>0$.   
This implies that $\|y_u^+\|_{H^1(\Omega)}=0$ and so  $\|y_u^+\|_{L^2(\Omega)}=0$. Hence $y_u^+=0$ for a.e. $x\in \Omega$ which leads to $y_u\leq 0$ for a.e. $x\in \Omega $. Since  $y_u\in C(\bar \Omega)$, we get $y_u\leq 0$ for all $x\in \bar \Omega$. Combining this with the assumption $(H4)$ yields  $g_y[\cdot]y_u\leq 0$. It follows that
\begin{align*}
g_{iy}[\cdot] y_u + u-f_{i1}=g_{iy}[\cdot] y_u+\sum_{j\neq i}u_j\in K_i.
\end{align*}
Define
$$d_i= g_i[\cdot]+ \bar u+  g_{iy}[\cdot] y_u +
u-f_i, \quad i=1,2,...,m.$$
Then $d_i\in K_i$, and 
$$f_i=g_{iy}[\cdot] y_u +u-(d_i- g_i[\cdot]- \bar u),$$
for all $i=1,2,...,m$, and $(y_u, u)\in T(D(\bar \lambda),\bar z)$. Hence,  \eqref{H3Assump}  is valid with $v_i=d_i- g_i[\cdot]-\bar u$ for $i=1,2,...,m$. Consequently, $(iv)$ is satisfied.
The lemma is proved. 
\end{proof}

According to \cite{Aubin} and \cite[Lemma 12]{Son}, we have an exact formula for computing the cones $T(K, G(\bar z, \bar \lambda))$ and $N(K, G(\bar z, \bar \lambda))$ which will be used in the proof of next lemma: 
\begin{align*}
	&T(K, G(\bar z, \bar \lambda))= \prod_{i=1}^{m}T(K_i, G_i(\bar z, \bar \lambda)), \\
	& N(K, G(\bar z, \bar \lambda))= \prod_{i=1}^{m}N(K_i, G_i(\bar z, \bar \lambda)). 
\end{align*} 
where
\begin{align*}
	T(K_i, G_i(\bar z, \bar \lambda))&=\left\{ w \in L^2(\Gamma) \mid w(x) \in T((-\infty,0], G_i(\bar z,\bar \lambda)(x)) \ \text{a.e.} \ x \in \Gamma \right\}\\
	&=\left\{ w \in L^2(\Gamma) \mid w(x) \leq 0 \quad   \text{whenever} \quad g_i[x]+\bar u(x)=0 \right \},
\end{align*} 
and 
\begin{align*}
	N(K_i, G_i(\bar z,\bar \lambda)) &= \left\{ w^* \in L^2(\Gamma) \mid w^*(x) \in N\big((-\infty,0], G_i(\bar z,\bar \lambda)(x)\big) \quad \text{a.e.} \quad x\in \Gamma  \right\} \\
	&= \left\{ w^* \in L^2(\Gamma) \Big |\  w^*(x) \begin{cases}
		\geq 0 &\quad \text{ whenever }\quad   g_i[x]+\bar u(x)=0 \\
		=0,  &\quad \text{otherwise}
	\end{cases}  \right\}.
\end{align*}
The solution stability in $L^2$-norm is established in the following lemma. The proof of this lemma is followed by the same technique as in \cite[Lemma 4.1]{Kien3}. For the convenience of the reader, we provide a proof here.
\begin{lemma} \label{key2}
Under the assumptions of Theorem 2.1, $(\bar y, \bar u)$ is a locally strongly optimal  solution of $(P(\bar \lambda))$  and there exist positive numbers $R_0, s_0$ and $l_0$ such that for all $\lambda \in B_{L^{\infty}(\Gamma)}(\bar \lambda, s_0)$,   every couple $(\hat y_\lambda, \hat u_\lambda)\in \mathcal S_{R_0}(\lambda)$ is a locally optimal solution of $(P(\lambda))$ and
\begin{align}\label{HolderCond3}
	\|\hat y_\lambda-\bar y\|_{L^\infty(\Omega)} + \|\hat u_\lambda-\bar u\|_{{L^2(\Gamma)}} \leq  l_0\|\lambda-\bar \lambda\|_{L^\infty(\Gamma)}^{1/2}.
\end{align}
\end{lemma}
\begin{proof} 
First, we show that $\bar z=(\bar y, \bar u)$ is a locally strongly optimal solution of $(P(\bar \lambda))$. By contradiction, suppose that this assertion is false. Then there is a sequence $\{z_k=(y_k, u_k)\}\subset \mathcal{F}(\bar \lambda)$, $z_k \to
\bar z$ such that
\begin{align}\label{I0}
 J(z_k, \bar \lambda)< J(\bar z, \bar \lambda)+o(t_k^2),
\end{align} where $t_k:=\|z_k- \bar z\|_Z \to 0$ as $k \to \infty$. Putting
$\hat z_k=(\hat y_k, \hat u_k), \hat y_k=
\frac{y_k-\bar y}{t_k}, \hat u_k= \frac{u_k-\bar u}{t_k}$ , we have  $\|\hat z_k \|_Z=\|\hat y_k\|_Y+\|\hat u_k\|_U=1$. Since $Z$ is reflexive, we may assume that $\hat z_k \rightharpoonup \hat z=(\hat y, \hat u)$. Moreover, since the embedding $Y=W^{1,r}(\Omega) \hookrightarrow C(\bar \Omega)$ is compact,
$\hat y_k \to \hat y$ in $C(\bar \Omega).$\\
\textit{Claim 1.} $\hat z\in \mathcal{C}[\bar z, \bar \lambda]$.

Writing $z_k=\bar z+t_k\hat z_k$ , it follows from a first-order Taylor expansion and \eqref{I0} that 
$$
\nabla_z J(\bar z, \bar \lambda)\hat z_k+\frac{o(t_k)}{t_k}\leq \frac{o(t_k^2)}{t_k}.
$$ Passing to the limit as $k\to \infty$, we obtain $\nabla_z J(\bar z, \bar \lambda)\hat z\leq 0$. 

Since $z_k \in \mathcal F(\bar \lambda)$, $H(z_k,\bar \lambda)=H(\bar z,\bar \lambda)=0$ for all $k\geq 1$. By a similar argument, we get $\nabla_z H(\bar z, \bar \lambda)\hat z=0$, which is equivalent to
$$
\begin{cases}
	A \hat y + h_y[\cdot]\hat y =0 &\text{ in } \Omega\\
	\partial_\nu \hat y=\hat u &\text{ on } \Gamma. \end{cases}
$$ 
Also,  since
$G(z_k, \bar \lambda)-G(\bar z, \bar \lambda)\in K-G(\bar z, \bar \lambda)$, we have
$$
\nabla_z G(\bar z, \bar \lambda) \hat z_k\in \frac{1}{t_k}
(K-G(\bar z, \bar \lambda))\subset T(K, G(\bar z, \bar \lambda)).
$$ Since $T(K, G(\bar z, \bar \lambda))$ is weakly closed, we obtain $\nabla_z G(\bar z, \bar \lambda)\hat z\in
T(K, G(\bar z, \bar \lambda))$.  This is equivalent to
\begin{align*}
&g_{iy}[x]\hat y(x) + \hat u(x)\geq 0\ \text{whenever}\  g_i[x]+\bar u(x)=0,\ i=1,2,...,m.
\end{align*} Hence $\hat y$ and $\hat u$ satisfy the conditions $(i)-(iii)$ of Definition \ref{DefCrD}, and so $(\hat y, \hat u)\in\mathcal{C}[\bar z, \bar \lambda]$.\\
\textit{Claim 2.}  $(\vartheta, e_1, e_2,...,e_m)\in \Lambda[(\bar y, \bar u),  \bar \lambda]$  is equivalent to the following conditions:
\begin{align*}%\label{StationaryCond}
\nabla_z \mathcal{L}(\bar z , \vartheta, e)=0, \quad e \in
N(K, G(\bar z, \bar \lambda)),\  e=(e_1, e_2,...,e_m),
\end{align*} where  Lagrangian $\mathcal{L}(z, \vartheta, e)$ is given by
\begin{align*}
\mathcal{L}(z, \vartheta, e)=J(z, \bar \lambda) +\langle \vartheta, H(z, \bar \lambda)\rangle + \langle e, G( z, \bar \lambda)\rangle.
\end{align*} In fact,  the condition $\nabla_z \mathcal{L}(\bar z , \vartheta, e)=0$ is rewritten as follows
\begin{align*}
&\begin{cases}
	-A \vartheta +h_y(\bar  y)\vartheta +L_y(\cdot,\bar  y)=0 &\text{ in } \Omega \\
	\partial_\nu \vartheta +\ell_y[\cdot]+\sum_{i=1}^{m}g_{iy}[\cdot]e_i=0  &\text{ on } \Gamma, 
\end{cases}  \\
&\alpha(\bar \lambda) +\beta(\bar \lambda)\bar u -\vartheta + \sum_{i=1}^{m}e_{\lambda i}=0,
\end{align*} which is equivalent to the conditions \eqref{adj} and \eqref{sta}; while the condition $e \in
N(K, G(\bar z, \bar \lambda))$ means
$
e_i(x)\in N((-\infty,0], g_i[x] +u(x))$ for a.e. $ x \in \Gamma,$
or 
$$ e_i(x) \begin{cases}
	\geq 0 &\quad \text{ whenever }\quad   g_i[x]+\bar u(x)=0 \\
	=0,  &\quad \text{otherwise}
\end{cases}$$ for all $i=1,2,...,m$. 
By a simple argument, the above condition is equivalent to
\begin{align*}
e_i(x)\geq 0,\ e_i(x)(g_i[x] +u(x))=0 \text{ a.e.}\ x\in\Gamma,\ i=1,2,...,m 
\end{align*} which is the condition \eqref{com}. Hence, the claim is justified. \\
\textit{Claim 3.} $\hat z=0$.

 By a second-order Taylor expansion for $\mathcal{L}$, one has
\begin{align*}
\mathcal{L}(z_k, \vartheta, e)-\mathcal{L}(\bar z, \vartheta,
e)&= t_k \nabla_z\mathcal{L}(\bar z, \vartheta, e)\hat z_k
+\frac{t_k^2}{2}\nabla^2_{zz}\mathcal{L}(\bar  z, \vartheta,
e)(\hat z_k,\hat z_k) +o(t_k^2)\\
&=0+\frac{t_k^2}{2}\nabla^2_{zz}\mathcal{L}(\bar z, \vartheta,
e)(\hat z_k, \hat z_k) + o(t_k^2).
\end{align*}
On the other hand, we get
\begin{align*}
&\mathcal{L}(z_k, \vartheta, e)-\mathcal{L}(\bar z, \vartheta,
e)\\
&=J(z_k, \bar \lambda)-J(\bar z, \bar \lambda)
+\langle \vartheta, H(z_k, \bar \lambda)-H(\bar z, \bar \lambda)\rangle+\langle e, G(z_k, \bar \lambda)-G(\bar z, \bar \lambda)\rangle\\
&=J(z_k, \bar \lambda)-J(\bar z, \bar \lambda) +\langle e,  G(z_k, \bar \lambda)-G(\bar z, \bar \lambda)\rangle\\
 & \leq J(z_k, \bar \lambda)-J(\bar z, \bar \lambda)\leq
o(t_k^2).
\end{align*} Here we used the fact that $e\in N(K, G(\bar z, \bar \lambda))$, $H(z_k,\bar \lambda)=H(\bar z,\bar \lambda)=0$ and \eqref{I0}. Therefore, we have
\begin{align}\label{iqL1}
\nabla^2_{zz}\mathcal{L}(\bar z, \vartheta, e)(\hat z_k, \hat
z_k) \leq \frac{o(t_k^2)}{t_k^2}.
\end{align} By letting $k\to\infty$,  we have
$$
\nabla^2_{zz}\mathcal{L}(\bar z, \vartheta, e)(\hat z, \hat
z) \leq 0.
$$ This is equivalent to
\begin{align*}
\int_\Omega &\big(L_{yy}[x]\hat y^2(x)+\vartheta h_{yy}[x]\hat y^2(x) \big)dx \notag \\&+ \int_\Gamma \big( \ell_{yy}[x]\hat y^2(x)+  \beta (\bar{\lambda}(x))\hat u^2(x)  + 
\sum_{i=1}^{m}e_i(x) g_{iyy}[x]\hat y^2(x) \big) ds \leq 0.
\end{align*} Combining this with \eqref{StrictSOC}, we obtain $\hat
z=0$.\\
\textit{Claim 4.} There exists $\sigma_0 >0$ such that $\|\hat u_k\|_U > \sigma_0$ for $k$ large enough. 

Since $ H(\bar z, \bar \lambda)=0$ and $\nabla_y H(\bar y, \bar u, \bar \lambda)$ is
bijective, the Implicit Function Theorem implies that there exist
balls $B_Y(\bar y, \epsilon_1)$, $B_U(\bar u, \epsilon_2)$
and a mapping $\phi: B_U(\bar u, \epsilon_2)\to B_Y(\bar y,
\epsilon_1)$ of class $C^1$ such that $H(\phi(u), u, \bar \lambda)=0$ for all
$u\in B_U(\bar u, \epsilon_2)$. In particular, $\phi$ is
Lipschitz continuous on $B_U(\bar u, \epsilon_2)$ with constant
$k_\phi>0$. Since $z_k\to \bar z$, there exists a positive integer
$k_0$ such that $y_k=\phi(u_k)$ for
$k>k_0$. By definition of $(\hat y_k, \hat u_k)$, we have
\begin{align*}
\|\hat y_k\|_Y=\frac{\|\phi(u_k)-\phi(\bar u)\|_Y}{t_k}\leq
\frac{k_\phi\|u_k-\bar u\|_U}{t_k}=k_\phi\|\hat u_k\|_U.
\end{align*}
It follows that
$$
1=\|\hat y_k\|_Y + \|\hat u_k\|_U \leq (1+ k_\phi)\|\hat u_k\|_U
$$  and so $\|\hat u_k\|_U\geq\sigma_0$ for all $k>k_0$, where $\sigma_0:=\frac{1}{1+ k_\phi}$. 

Combining this with  \eqref{iqL1} and $(H4)$ yields
 \begin{align}
\frac{o(t_k^2)}{t_k^2}\geq& \int_\Omega \big(L_{yy}[x]\hat y_k^2(x)+\vartheta h_{yy}[x]\hat y_k^2(x) \big)dx \notag \\&\quad+ \int_\Gamma \big( \ell_{yy}[x]\hat y_k^2(x)+  \beta (\bar{\lambda}(x))\hat u_k^2(x)  + 
\sum_{i=1}^{m}e_i(x) g_{iyy}[x]\hat y_k^2(x) \big) ds \notag\\
\geq& \int_\Omega \big(L_{yy}[x]\hat y_k^2(x)+\vartheta h_{yy}[x]\hat y_k^2(x) \big)dx \notag \\&\quad+ \int_\Gamma \big( \ell_{yy}[x]\hat y_k^2(x)+  \gamma\sigma_0^2  + 
\sum_{i=1}^{m}e_i(x) g_{iyy}[x]\hat y_k^2(x) \big) ds \label{iqL2}
\end{align} 
Notice that  $\hat y_k\to 0$ in
$C(\bar\Omega)$. By letting $k\to\infty$, it follows from \eqref{iqL2} 
that $0\geq \gamma\sigma_0^2$, which is impossible.  
Therefore, $\bar z$ is a locally strongly optimal solution of $(P(\bar \lambda))$.

Now we will use Proposition 3.1 in \cite{Kien3}. From Lemma \ref{key1} and the fact $\bar z$ is a locally strongly optimal solution of $(P(\bar \lambda))$, it follows that the assumptions of Proposition 3.1 in \cite{Kien3} are fulfilled. Therefore,  there exist positive numbers $R_0,  s_0$ and $l_0$  such that for all $\lambda \in B(\bar \lambda, s_0)$, every couple  $(\hat  y_\lambda, \hat  u_\lambda)\in	\mathcal S_{R_0}(\lambda)$  is a locally optimal solution of $(P(\lambda))$, and  
\begin{align*}%\label{HolderCond3}
	\mathcal S_{R_0}(\lambda)\subset (\bar y, \bar u) +  l_0\|\lambda- \lambda_0\|_\lambda^{1/2}\bar B_Z,
\end{align*} 
which is equivalent to \eqref{HolderCond3}.  
The proof of the lemma is complete.
\end{proof}

 Notice that, according  Lemma 3.1 in \cite{Kien3}, there exist  numbers $\beta_0>0$ and   $M_0>0$ such that
\begin{align}\label{Bound1}
\|(e_{\lambda 1},..., e_{\lambda m})\|_{(L^2(\Gamma))^m} +\|\vartheta_\lambda\|_{W^{1,s}(\Omega)}\leq M_0
\end{align} for all $\lambda \in B(\bar \lambda, \beta_0),\ (e_{\lambda 1},..., e_{\lambda m}, \vartheta_\lambda)\in \Lambda[(\hat  y_\lambda, \hat u_\lambda), \lambda]$. Besides, by assumption $(H3)$, there is a number $\beta_1>0$ such that for a.e. $x\in\Gamma$, one has
\begin{align}\label{StrictNonegative}
\beta(\lambda(x))>\frac{\gamma}2,\  \forall \lambda \in B(\bar \lambda, \beta_1).
\end{align} Without loss of generality we can assume that \eqref{HolderCond1},  \eqref{Bound1} and \eqref{StrictNonegative} are valid for all $\lambda \in B(\bar \lambda, s_0)$

\begin{lemma}\label{LemmaReg} For each $\lambda \in B(\bar \lambda, s_0)$,  if $(\hat  y_\lambda, \hat  u_\lambda)\in	\mathcal S_{R_0}(\lambda)$ then $\hat u_\lambda \in L^\infty(\Gamma)$.
\end{lemma}
\begin{proof} Fix $\lambda \in B(\bar \lambda, s_0)$. Since  $(\hat  y_\lambda, \hat  u_\lambda)\in	\mathcal S_{R_0}(\lambda)$  is a locally optimal solution of $(P(\lambda))$, there exist Lagrange multipliers $\vartheta_\lambda\in W^{1,s}(\Omega),\  e_{\lambda 1}, e_{\lambda 2}, ...,e_{\lambda m}\in L^2(\Gamma)$ satisfying conditions \eqref{adj}-\eqref{com}.
It follows from conditions \eqref{adj} and \eqref{sta} that for a.e. $x \in \Gamma$,
\begin{align*}
& \alpha(\lambda(x))+\beta(\lambda(x))\hat  u_\lambda(x)-\vartheta_\lambda(x)+ \sum_{i=1}^m e_{\lambda i}(x)=0\\ 
&e_{\lambda i}(x)\in N((-\infty,0],  g_i(x, \hat y_\lambda(x), \lambda(x))+ \hat u_\lambda(x))\quad \forall i=1,2,...,m. 
\end{align*} 
Hence, setting $\xi(x):=\sum_{i=1}^m e_{\lambda i}(x)$, for a.e. $x\in \Gamma$ one has  
\begin{align*}
&\xi(x)=	-\alpha(\lambda(x))-\beta(\lambda(x))\hat  u_\lambda(x)+\vartheta_\lambda(x);\\
&\xi(x)- \sum_{j\neq i}e_{\lambda j}(x)\in N((-\infty,0],    g_i(x, \hat y_\lambda(x), \lambda(x))+ \hat u_\lambda(x))), \quad \forall i=1,2,...,m.   
\end{align*} 
Let $g(x, \hat y_\lambda(x), \lambda(x))=\max \limits_{1\leq i\leq m}{ g_i(x, \hat y_\lambda(x), \lambda(x))}$
 and $G_i[x,\lambda]=g_i(x, \hat y_\lambda(x), \lambda(x))+\hat u_\lambda(x)$ for all $x\in \Gamma$. We claim that 
 \begin{equation*}%\label{MetricProj0}
 	g(x, \hat y_\lambda(x), \lambda(x)) +\hat u_\lambda(x)=P_{(- \infty,0]} \left( \frac{1}{\beta(\lambda(x))}(\vartheta_\lambda(x)-\alpha(\lambda(x))) + g(x, \hat y_\lambda(x), \lambda(x))\right)
 \end{equation*} for a.e. $x\in \Gamma$, where $P_{(- \infty,0]}(a)$ is the metric projection of $a$ onto $(- \infty,0]$.

For any $i\in \{1,2,...,m\}$ and $\eta\in (- \infty,0]$, we have
$$
(\xi(x)- \sum_{j\neq i} e_{\lambda j}(x))(\eta -G_i[x,\lambda])\leq 0,
$$ which implies that
\begin{align*}
\xi(x)(\eta-G_i[x,\lambda])&\leq \sum_{j\neq i} e_{\lambda j}(x)(\eta -G_i[x,\lambda])\\
&\leq \sum_{j\neq i} e_{\lambda j}(x)(\eta -G_j[x,\lambda]) + \sum_{j\neq i} e_{\lambda j}(x))(G_j[x,\lambda] -G_i[x,\lambda]).
\end{align*} Notice that $e_{\lambda i}(x)\geq 0$ $\forall i=1,2,...,m$. Therefore, if $x\in \{x'\in \Gamma: g_i(x',\hat y_\lambda(x'),\lambda (x'))=g(x',\hat y_\lambda(x'),\lambda (x'))\}$ then one has
$$\xi(x)(\eta-G_i[x,\lambda])\leq  0,$$ 
which entails that $$
\xi(x)\in N((-\infty;0], G_i[x,\lambda]),
$$ and so,   
\begin{align*}
\vartheta_\lambda(x)-\alpha(\lambda(x))-\beta(\lambda(x))\hat u_\lambda(x)\in N((- \infty,0], g_i(x, \hat y_\lambda(x), \lambda(x)) +\hat u_\lambda(x))\end{align*}
or 
\begin{align}\label{NormalCone1}
	\vartheta_\lambda(x)-\alpha(\lambda(x))-\beta(\lambda(x))\hat u_\lambda(x)\in N((- \infty,0], g(x, \hat y_\lambda(x), \lambda(x)) +\hat u_\lambda(x)).\end{align} 
Since $\Gamma= \cup_{1\leq i\leq m} \{x\in \Gamma: g_i(x,\hat y_\lambda(x),\lambda (x))=g(x,\hat y_\lambda(x),\lambda (x))\}$, the formula \eqref{NormalCone1} is valid for all $x\in \Gamma$. Hence, we conclude that
$$
\frac{1}{\beta(\lambda(x))}(\vartheta_\lambda(x)-\alpha(\lambda(x))) -\hat u_\lambda(x)\in N((- \infty,0], g(x, \hat y_\lambda(x), \lambda(x)) +\hat u_\lambda(x))
$$ or equivalently
\begin{align*}
&\frac{1}{\beta(\lambda(x))}(\vartheta_\lambda(x)-\alpha(\lambda(x))) + g(x, \hat y_\lambda(x), \lambda(x))-(g(x, \hat y_\lambda(x), \lambda(x))+\hat u_\lambda(x))\\
&\quad  \quad \quad \in N((- \infty,0], g(x, \hat y_\lambda(x), \lambda(x)) +\hat u_\lambda(x)) \quad \text{a.e. }x\in \Gamma.
\end{align*} By \cite[Theorem 5.2]{Brezis1}, for a.e. $x\in \Gamma$ we obtain
\begin{equation}\label{MetricProj1}
g(x, \hat y_\lambda(x), \lambda(x)) +\hat u_\lambda(x)=P_{(- \infty,0]}\left( \frac{1}{\beta(\lambda(x))}(\vartheta_\lambda(x)-\alpha(\lambda(x))) + g(x, \hat y_\lambda(x), \lambda(x))\right).
\end{equation} 
Besides, from  \eqref{HolderCond1} and the embedding $W^{1,r}(\Omega)\hookrightarrow C(\bar \Omega)$, it follows that 
\begin{align*}
\|\hat y_\lambda-\bar y\|_{L^\infty(\bar\Omega)}\leq C\|\hat y_\lambda-\bar y\|_{W^{1,r}(\bar\Omega)}\leq Cl_0\|\lambda-\bar \lambda\|^{1/2}\leq Cl_0 s_0^{1/2}.
\end{align*} Hence
$$
\|\hat y_\lambda\|_{L^\infty(\bar\Omega)}\leq  \|\bar y\|_{L^\infty(\bar\Omega)}+ Cl_0 s_0^{1/2}.
$$ Also, one has
$$
\|\lambda\|_{L^\infty(\Gamma)} \leq \|\bar \lambda\|_{L^\infty(\Gamma)} +\|\lambda-\bar \lambda\|_{L^\infty(\Gamma)} \leq \|\bar \lambda\|_{L^\infty(\Gamma)}+ s_0.
$$ 
Putting
$
M=C l_0 s_0^{1/2} + \|\bar y\|_{L^\infty(\bar\Omega)}+ s_0 + \|\bar \lambda\|_\infty$,
we have
\begin{align}\label{BoundM1}
	\|\bar y\|_{L^\infty(\bar\Omega)} + \|\bar \lambda\|_{L^\infty(\Gamma)}\leq M,\ \|\hat y_\lambda\|_{L^\infty(\bar\Omega)} +\|\lambda\|_{L^\infty(\Gamma)}\leq M.
\end{align} 
It follows from assumption $(H_2)$ and \eqref{BoundM1} that there exists $k_{g_iM}$ such that 
\begin{align*}
	|g_i(x, \hat y_\lambda(x), \lambda(x)|&\leq |g_i(x, \hat y_\lambda(x), \lambda(x)- g_i(x, 0, 0)| + |g_i(x, 0, 0)| \notag\\
	&\leq k_{g_iM}(|\hat y_\lambda(x)| + |\lambda(x)|)+ |g_i(x, 0, 0)|\leq k_{g_iM}M+ |g_i(x, 0, 0)|  
\end{align*}
for all $i=1,2,...,m$. Putting $k_{gM}=\max \limits_{1\leq i\leq m}{ k_{g_iM}}$, we have $$
	|g_i(x, \hat y_\lambda(x), \lambda(x)|\leq k_{gM}M+ |g_i(x, 0, 0)|  
$$
for all $i=1,2,...,m$. By a similar argument, the assumptions $(H_1)$, $(H3)$ and \eqref{BoundM1} implies that there exist $k_{LM}, k_{\ell M}, k_{gM}, k_{\alpha M}$ such that 
\begin{align*}
	&|L_y(x,\hat y_\lambda(x))|\leq k_{LM}M+|L(x,0)|, \notag\\
	&|\ell(x',\hat y_\lambda(x'))|\leq k_{\ell M}M+|\ell(x',0)|,\notag \\
	& |g_{iy}(x', \hat y_\lambda(x'),\lambda(x'))|\leq  k_{gM} M + |g_{iy}(x', 0, 0)|, \\
	&|\alpha(\lambda(x'))|\leq  k_{\alpha M}M +|\alpha(0)| \notag
\end{align*} for a.e. $x\in \Omega, x'\in \Gamma$,  for all $i=1,2,...,m$. Hence $L_y(\cdot,\hat y_\lambda) \in L^\infty(\Omega)$ and $\ell(\cdot,\hat y_\lambda),\ g_{i}(\cdot, \hat y_\lambda,\lambda),\\ g_{iy}(\cdot, \hat y_\lambda,\lambda),\ \alpha(\lambda)\in L^\infty(\Gamma)$ and so, $g(\cdot, \hat y_\lambda, \lambda)\in L^\infty(\Gamma)$ and $\sum \limits_{i=1}^{m} g_{iy}(\cdot, \hat y_\lambda,\lambda)e_{\lambda i}\in L^2(\Gamma)$. Combining this with   \eqref{adj} and $(H4)$ yields $\vartheta_\lambda \in W^{1,r}(\Omega)$ and so, $\vartheta_\lambda \in C(\bar \Omega)$. 

Using the non-expansive property of metric projection (see \cite[Proposition 5.3]{Brezis1})  and the fact  $P_{(- \infty,0]}(0)=0$, it follows from \eqref{MetricProj1} that 
\begin{align*}
|\hat u_\lambda(x)|&\leq  |g(x, \hat y_\lambda(x), \lambda(x))| +\left|\frac{1}{\beta(\lambda(x))}(\vartheta_\lambda(x)-\alpha(\lambda(x))) + g(x, \hat y_\lambda(x), \lambda(x))\right| \\
&\leq  |g(x, \hat y_\lambda(x), \lambda(x))| +\frac{1}{\gamma}(|\vartheta_\lambda(x)| +|\alpha(\lambda(x))| ) + |g(x, \hat y_\lambda(x), \lambda(x))| .
\end{align*} Hence,
\begin{align*}
|\hat u_\lambda(x)|\leq 2|g(x, \hat y_\lambda(x), \lambda(x))|  + \frac{1}{\gamma}(|\vartheta_\lambda(x)| +|\alpha(\lambda(x))| ) 
\end{align*}  
for a.e. $x\in \Gamma$.
Since  $g(\cdot, \hat y_\lambda, \lambda), \  \tau\vartheta_\lambda,\ \alpha(\lambda)\in L^\infty(\Gamma)$, we obtain $\hat u_\lambda \in L^\infty(\Gamma)$. The lemma is proved.
\end{proof}

It remains to prove \eqref{HolderCond2} of Theorem 2.1.  For this we have the following lemma. 

\begin{lemma} If the assumption $(H5)$ is valid then  there exist positive numbers  $l_1$ and $s_1$ such that for all $\lambda \in B_{L^{\infty}(\Gamma)}(\bar \lambda, s_1),$  and any $(\hat y_\lambda, \hat u_\lambda)\in \mathcal S_{R_0}(\lambda)$ one has
$$
\|\hat y_\lambda-\bar y\|_{W^{1,r}(\Omega)} + \|\hat u_\lambda-\bar u\|_\infty \leq  l_1\|\lambda-\bar \lambda\|_{L^\infty(\Gamma)}^{1/2}.
$$
\end{lemma}
\begin{proof} Let  $\lambda \in B_{L^\infty(\Gamma)}(\bar \lambda, s_0)$, and  $( \hat y_\lambda,  \hat u_\lambda)\in \mathcal S_{R_0}(\lambda)$, and $M=C l_0 s_0^{1/2} + \|\bar y\|_{L^\infty(\bar\Omega)}+ s_0 + \|\bar \lambda\|_\infty$.  Analysis similar to that in the proof of Lemma \ref{LemmaReg} shows that 
\begin{align}\label{BoundM}
\|\bar y\|_{L^\infty(\bar\Omega)} + \|\bar \lambda\|_{L^\infty(\Gamma)}\leq M,\ \|\hat y_\lambda\|_{L^\infty(\bar\Omega)} +\|\lambda\|_{L^\infty(\Gamma)}\leq M
\end{align} 
and $g_i(\cdot, \hat y_\lambda, \lambda)\in L^\infty(\Gamma)$ for all $i=1,2,...,m$. In particular, $g_i[\cdot]\in L^\infty(\Gamma)$ and so, $\operatorname*{ess~sup}\limits_{x\in\Gamma_i}(g_k[x]-g_i[x])>-\infty$ for all $1\leq i, k\leq m,\ k\neq i.$ Consequently, the assumption (H5) implies that there exists a number $\sigma_1>0$ such that 
$$g_k[x] - g_i[x] \leq -\sigma_1 \ \text{a.e. } x \in \Gamma_i,$$ 
which gives 
\begin{align}\label{bdt1}
G_k(\bar y, \bar u, \bar \lambda) \leq G_i(\bar y, \bar u, \bar \lambda)-\sigma_1 \leq -\sigma_1 <0\ \text{a.e. } x \in \Gamma_i	
\end{align} for all $1\leq i, k\leq m,\ k\neq i$, where $G_j(y, u, \lambda)=g_j(\cdot,y,\lambda)+u$ with $j=1,2,...,m$.   

It follows from \eqref{BoundM} that there exist Lipschitz constants  $k_{LM}$, $k_{\ell M}, k_{hM}, k_{gM}, k_{\alpha M}$ and $k_{\beta M}$ under which  $(H1), \ (H2)$ and $(H3)$ hold.  
From the assumption (H2), we have
\begin{align*}
	|g_i[x]- g_i(x, \hat y_\lambda(x), \lambda(x))| &\leq k_{gM}( |\hat y_\lambda(x)-\bar y(x)| + |\lambda(x)-\bar \lambda(x)|)\\
	& \leq k_{gM}\big(C l_0 \|\lambda-\bar \lambda\|_{L^\infty(\Gamma)}^{1/2} + \|\lambda-\bar \lambda\|_{L^\infty(\Gamma)}\big)\\
	& \leq k_{gM}\big(C l_0 \|\lambda-\bar \lambda\|_{L^\infty(\Gamma)}^{1/2} + s_0^{1/2}\|\lambda-\bar \lambda\|^{1/2}_\infty\big)\\
	&\leq k_{gM}\big(Cl_0 + s_0^{1/2}\big)\|\lambda-\bar \lambda\|_{L^\infty(\Gamma)}^{1/2}<\frac{\sigma_1}{4}
\end{align*} for a.e. $x\in \Omega$ and for all $i=1,2,...,m$, whenever $\lambda \in B(\bar \lambda, s_1)$, where
$$
s_1:=\min\left\{\frac{\sigma_1^2}{16k_{gM}^2(Cl_0 + s_0^{1/2})^2},s_0\right\}.$$
Therefore, for all $i,k \in \{1,2,...,m\}$, $k\neq i$, $\lambda \in B(\bar \lambda, s_1)$, we have 
\begin{align*}
	&g_k(x,\hat y_\lambda(x),\lambda(x)) -g_i(x,\hat y_\lambda(x), \lambda(x))\\
	&= (g_k(x,\hat y_\lambda(x),\lambda(x)) -g_k[x])+(g_k[x]-g_i[x])
	+(g_i[x] -g_i(x,\hat y_\lambda(x), \lambda(x)))\\
	&\leq |g_k(x,\hat y_\lambda(x),\lambda(x)) -g_k[x]|+(g_k[x]-g_i[x])
	+|g_i[x] -g_i(x,\hat y_\lambda(x), \lambda(x))|\\
	&\leq  \frac{\sigma_1}{4}-\sigma_1+\frac{\sigma_1}{4} =-\frac{\sigma_1}{2} \quad \text{a.e. }x\in \Gamma_i.
\end{align*}
This implies that 
\begin{align}	\label{bdt2}
G_k(\hat y_\lambda, \hat u_\lambda, \lambda) \leq G_i(\hat y_\lambda, \hat u_\lambda, \lambda)-\frac{\sigma_1}{2}\leq -\frac{\sigma_1}{2} <0 \quad \text{a.e. in } \Gamma_i\end{align}
for all $1\leq i,k\leq m$, $k\neq i$, $\lambda \in B(\bar \lambda, s_1)$.

 Fix $\lambda \in B(\bar \lambda, s_1)$ and $(\hat y_\lambda, \hat u_\lambda)\in \mathcal S_{R_0}(\lambda)$. Since $( \hat y_\lambda,  \hat u_\lambda)$ and $(\bar y, \bar u)$ are  locally optimal solutions of $(P(\lambda))$ and $(P(\bar \lambda))$, respectively, there exist the Lagrange multipliers $(\vartheta_\lambda, e_{\lambda 1}, e_{\lambda 2},...,e_{\lambda m})\in\Lambda[( \hat y_\lambda,  \hat u_\lambda), \lambda]$ and  $(\vartheta, e_1, e_2, ...,e_m)\in \Lambda[(\bar y, \bar u), \bar \lambda]$ satisfying   
\begin{align}
&\begin{cases} 
	&A \vartheta_\lambda +h_y(\cdot, \hat y_\lambda)\vartheta_\lambda = -L_y(\cdot,\hat y_ \lambda) \text{ in } \Omega \\
 &\partial_\nu \vartheta_\lambda = -\ell_y(\cdot,\hat y_\lambda,\lambda)-\sum \limits_{i=1}^{m}g_{iy}(\cdot, \hat y_\lambda,\lambda)e_{\lambda i} \text{ on } \Gamma, \end{cases}\label{FOC1}\\
 &\begin{cases} 
 	&A \vartheta +h_y[\cdot]\vartheta = -L_y[\cdot] \text{ in } \Omega \\
 	&\partial_\nu \vartheta = -\ell_y[\cdot]-\sum \limits_{i=1}^{m}g_{iy}[\cdot]e_{i} \text{ on } \Gamma, \end{cases} \label{FOC2}\\
-&\vartheta_\lambda +\alpha(\lambda)+\beta(\lambda) \hat u_\lambda +\sum \limits_{i=1}^{m}e_{\lambda i}=0,\label{FOC3}\\
-&\vartheta +\alpha(\bar \lambda)+\beta(\bar \lambda)\bar  u +\sum \limits_{i=1}^{m}e_{i}=0,\label{FOC4}\\
&e_{\lambda i}(x)(g_i(x,  \hat y_\lambda(x), \lambda (x))+\hat u_\lambda(x))=0, \quad e_{\lambda i}(x) \geq 0, \label{FOC5}\\
&e_i(x)(g_i[x]+ \bar u(x))=0, \quad e_i(x) \geq 0, \label{FOC7}\\
&g_i(x,  \hat y_\lambda(x), \lambda (x))+ \hat u_\lambda(x)\leq 0, \label{FOC9} \\
&g_i[x] +\bar  u(x)\leq 0 \label{FOC10}
\end{align} 
for a.e. $x\in\Gamma$, for all $i=1,2,...,m.$

Once again, analysis similar to that in the proof of Lemma \ref{LemmaReg} shows that $\vartheta_\lambda \in W^{1,r}(\Omega)$. Moreover, by \eqref{state}, we have  
\begin{align*}
\|\vartheta_\lambda\|_{W^{1,r}(\Omega)}&\leq C_1\big(\|L_y(\cdot,\hat y_\lambda)\|_{L^2(\Omega)}+\|\ell_y(\cdot,\hat y_\lambda,\lambda)+\sum \limits_{i=1}^{m} g_{iy}(\cdot, \hat y_\lambda,\lambda)e_{\lambda i}\|_{ L^2(\Gamma)}\big)\\
&\leq C_1\Big(\big(k_{LM}M+\|L(\cdot,0)\|_{L^\infty(\Omega)}\big)|\Omega|^{1/2}+\big(k_{\ell M}M+\|\ell(x,0,0)\|_{L^\infty(\Gamma)}\big)|\Gamma|^{1/2}\\
& \qquad \qquad \qquad \qquad +\big(mk_{g M}M+\sum \limits_{i=1}^{m} \|g_{iy}(\cdot,0,0)\|_{L^\infty(\Gamma)}\big)M_0  \Big):=M_1.
\end{align*}
In particular, $\vartheta \in W^{1,r}(\Omega)$ and $\|\vartheta\|\leq M_1$. Notice that $M_1$ do not depend on $\lambda$. From this and the embedding $W^{1,r}(\Omega)\hookrightarrow C(\bar \Omega)$, it follows that
there exists a positive number $M_2$ independent of $\lambda$ such that $$\|\vartheta_\lambda\|_{C(\bar\Omega)}+\|\vartheta\|_{C(\bar\Omega)}+\|\vartheta_\lambda\|_{L^2(\Omega)}+ \|\tau \vartheta_\lambda\|_{L^\infty(\Gamma)}+ \|\tau \vartheta\|_{L^\infty(\Gamma)}<M_2.$$  
 Besides, it follows from \eqref{bdt1}, \eqref{bdt2}, \eqref{FOC5} and \eqref{FOC7} that $e_k=e_{\lambda k}=0$ a.e. in $\Gamma_i$ for all $1\leq i,k\leq m$, $k\neq i$.
 	Combining this with \eqref{FOC3} and \eqref{FOC4} we obtain
 \begin{align*}
 	e_i=\vartheta-\alpha(\bar \lambda)-\beta(\bar \lambda)\bar u \quad \text{and} \quad  e_{\lambda i}=\vartheta_\lambda -\alpha(\lambda)-\beta(\lambda) \hat u_\lambda\ \text{ a.e. in } \Gamma_i. 
 \end{align*} 
 Hence 
 \begin{align}
 	e_i=(\vartheta-\alpha(\bar \lambda)-\beta(\lambda)\bar u))\chi_i \ \text{ and } \   e_{\lambda i}=(\vartheta_\lambda -\alpha(\lambda)-\beta(\lambda) \hat u_\lambda)\chi_i\ \text{ a.e. in} \ \Gamma, \label{eq31}
 \end{align}
 where $\chi_i$ is the characteristic function of the set $\Gamma_i$, $i=1,2,...,m$, that is 
 \begin{align*}
 	\chi_i(x)= \begin{cases}
 		1 \quad \text{if }x\in \Gamma_i,\\
 		0 \quad \text{if }x\notin \Gamma_i. 
 	\end{cases}
 \end{align*} 
 Substituting \eqref{eq31} into \eqref{FOC1} and \eqref{FOC2}, we see that 
 \begin{align*}
 	&\begin{cases} A\vartheta_\lambda  + h_y(\cdot,\hat y_\lambda)\vartheta_\lambda = -L_y(\cdot,\hat y_\lambda)\ {\rm in}\ \Omega,\\ 
 		\partial_\nu \vartheta_\lambda +\sum\limits_{i=1}^mg_{iy}(\cdot,\hat y_\lambda,\lambda)\chi_i\vartheta_\lambda=-\ell_y(\cdot,\hat y_\lambda,\lambda)+ \sum \limits_{i=1}^mg_{iy}(\cdot,\hat y_\lambda,\lambda) (\alpha(\lambda)+\beta(\lambda) \hat u_\lambda)\chi_i   \ {\rm on}\ \Gamma, 
 	\end{cases}\\%\label{eq41}  
 &\begin{cases} A\vartheta  + h_y[\cdot]\vartheta= -L_y[\cdot]\ {\rm in}\ \Omega,\\ 
 	\partial_\nu \vartheta +\sum\limits_{i=1}^mg_{iy}[\cdot]\chi_i\vartheta=-\ell_y[\cdot]+ \sum \limits_{i=1}^mg_{iy}[\cdot](\alpha(\bar \lambda)+\beta(\bar \lambda) \bar u)\chi_i   \ {\rm on}\ \Gamma. 
 \end{cases}%\label{eq42}
 \end{align*}
By subtracting the above equations, we get
\begin{align*}
	\begin{cases} &A(\vartheta-\vartheta_\lambda)  + h_y[\cdot ](\vartheta-\vartheta_\lambda) = L_y(\cdot,\hat y_\lambda )-L_y[\cdot]+(h_y(\cdot,\hat y_\lambda)-h_y[\cdot])\vartheta_\lambda \ {\rm in}\ \Omega,\\ 
		&\partial_\nu (\vartheta-\vartheta_\lambda) +\sum\limits_{i=1}^mg_{iy}[\cdot]\chi_i(\vartheta-\vartheta_\lambda)=\ell_y(\cdot,\hat y_\lambda,\lambda)-\ell_y[\cdot]+\sum\limits_{i=1}^m(g_{iy}(\cdot,\hat y_\lambda,\lambda)-g_{iy}[\cdot])\chi_i\vartheta_\lambda
		\\
		&\qquad \qquad \quad +\sum \limits_{i=1}^mg_{iy}[\cdot](\alpha(\bar \lambda)+\beta(\bar \lambda) \bar u)\chi_i-\sum \limits_{i=1}^mg_{iy}(\cdot,\hat y_\lambda,\lambda) (\alpha(\lambda)+\beta(\lambda) \hat u_\lambda)\chi_i\  {\rm on}\ \Gamma. 
	\end{cases}
\end{align*}
By (H4) and \eqref{state}, we get
\begin{align}
&\|\vartheta-\vartheta_\lambda\|_{C(\bar \Omega)}\leq  C\|\vartheta-\vartheta_\lambda\|_{W^{1,r}(\Omega)}\leq CC_1\big(\|I_\Omega\|_{L^2(\Omega)}+\|I_\Gamma\|_{L^2(\Gamma)}\big), \label{EtimAdjointSol}
\end{align}
where 
\begin{align*}
	&I_\Omega:=L_y(\cdot,\hat y_\lambda )-L_y[\cdot]+(h_y(\cdot,\hat y_\lambda)-h_y[\cdot])\vartheta_\lambda\\
	&I_\Gamma:=\ell_y(\cdot,\hat y_\lambda,\lambda)-\ell_y[\cdot]+\sum\limits_{i=1}^m(g_{iy}(\cdot,\hat y_\lambda,\lambda)-g_{iy}[\cdot])\chi_i\vartheta_\lambda
	\\
	&\qquad \qquad \quad +\sum \limits_{i=1}^mg_{iy}[\cdot](\alpha(\bar \lambda)+\beta(\bar \lambda) \bar u)\chi_i-\sum \limits_{i=1}^mg_{iy}(\cdot,\hat y_\lambda,\lambda) (\alpha(\lambda)+\beta(\lambda) \hat u_\lambda)\chi_i.
	\end{align*}
We have 
\begin{align*}
\|I_\Omega \|_{L^2(\Omega)}&\leq \|(h_y(\cdot, \hat y_\lambda)-h_y[\cdot])\vartheta_\lambda\|_{L^2(\Omega)} +\|L_y(\cdot, \hat y_\lambda)- L_y[\cdot] \|_{L^2(\Omega)}\\
&\leq \|h_y[\cdot]-h_y(\cdot, \hat y_\lambda)\|_{L^\infty(\Omega)} \|\vartheta_\lambda\|_{L^2(\Omega)} +\|L_y(\cdot, \hat y_\lambda)- L_y[\cdot] \|_{L^\infty(\Omega)} |\Omega|^{1/2}\\
&\leq  k_{hM}\|\hat y_\lambda-\bar y\|_{L^\infty(\Omega)} M_2 + |\Omega|^{1/2}k_{LM}\|\hat y_\lambda-\bar y\|_{L^\infty(\Omega)}\\
&\leq (k_{hM} M_2+|\Omega|^{1/2}k_{LM})C\|\hat y_\lambda-\bar y\|_{W^{1,r}(\Omega)}\\
&\leq (k_{hM} M_2+|\Omega|^{1/2}k_{LM})Cl_0\|\lambda- \bar \lambda\|_{L^\infty(\Gamma)}^{1/2}.
\end{align*} 
Hence
\begin{align}\label{AdjointEstim1}
\|I_\Omega \|_{L^2(\Omega)}\leq k_1\|\lambda-\bar \lambda\|_{L^\infty(\Gamma)}^{1/2},
\end{align}where $k_1:=(k_{hM} M_2+|\Omega|^{1/2}k_{LM})Cl_0$. It remains to estimate $\|I_\Gamma\|_{L^2(\Gamma)}$.  For this, we have 
\begin{align}\label{EstimCase2}
\|I_\Gamma\|_{L^2(\Gamma)}&\leq \| \ell_y(\cdot, \hat y_\lambda, \lambda)-\ell_y[\cdot] \|_{L^2(\Gamma)} + \left\| \sum\limits_{i=1}^m(g_{iy}(\cdot,\hat y_\lambda,\lambda)-g_{iy}[\cdot])\chi_i\vartheta_\lambda\right\|_{L^2(\Gamma)}\notag \\
&+\left\|\sum \limits_{i=1}^mg_{iy}[\cdot](\alpha(\bar \lambda)+\beta(\bar \lambda) \bar u)\chi_i-\sum \limits_{i=1}^mg_{iy}(\cdot,\hat y_\lambda,\lambda) (\alpha(\lambda)+\beta(\lambda) \hat u_\lambda)\chi_i\right\|_{L^2(\Gamma)}.
\end{align} We now give estimates for terms on the right hand side of \eqref{EstimCase2}.  For the first term, using   \eqref{HolderCond1}, we have
\begin{align}\label{BoundTerm1}
	\| \ell_y[\cdot]-\ell_y(\cdot, \hat y_\lambda, \lambda) \|_{L^2(\Gamma)} &\leq \| \ell_y[\cdot]-\ell_y(\cdot, \hat y_\lambda, \lambda) \|_{L^\infty(\Gamma)} |\Gamma|^{1/2} \notag
	\\
	&\leq k_{\ell M}(\|\bar y - \hat y_\lambda  \|_{L^\infty(\Gamma)}+\|\bar \lambda - \lambda  \|_{L^\infty(\Gamma)}) |\Gamma|^{1/2} \notag\\
	&\leq k_{\ell M}\big(Cl_0\|\bar \lambda - \lambda  \|^{1/2}_{L^\infty(\Gamma)}+s_1^{1/2}\|\bar \lambda - \lambda  \|^{1/2}_{L^\infty(\Gamma)}\big) |\Gamma|^{1/2} \notag\\
	&\leq k_2 \|\lambda-\bar \lambda\|_{L^\infty(\Gamma)}^{1/2},
\end{align} where $k_2:=k_{\ell M}(Cl_0+s_1^{1/2})|\Gamma|^{1/2}$.  Similarly, for the second term, one has 
\begin{align}\label{BoundTerm2}
	 \left\| \sum\limits_{i=1}^m(g_{iy}(\cdot,\hat y_\lambda,\lambda)-g_{iy}[\cdot])\chi_i\vartheta_\lambda\right\|_{L^2(\Gamma)}
	&	\leq\sum\limits_{i=1}^m \|g_{iy}(\cdot,\hat y_\lambda,\lambda)-g_{iy}[\cdot]\|_{L^\infty(\Gamma)}\|\chi_i\vartheta_\lambda\|_{L^2(\Gamma)}\notag\\
	&\leq\sum\limits_{i=1}^m k_{gM}\left(\|\hat y_\lambda-\bar y\|_{L^\infty(\Omega)}+ \|\lambda-\bar \lambda\|_{L^\infty(\Gamma)}\right)\|\chi_i\vartheta_\lambda\|_{L^2(\Gamma)}\notag\\
	&\leq k_{gM}\left(\|\hat y_\lambda-\bar y\|_{L^\infty(\Omega)}+ \|\lambda-\bar \lambda\|_{L^\infty(\Gamma)}\right)\|\vartheta_\lambda\|_{L^2(\Gamma)}\notag\\
	&\leq k_{gM}\left( Cl_0\|\lambda-\bar \lambda\|^{1/2}_{L^\infty(\Gamma)}+s_1^{1/2}\|\lambda-\bar \lambda\|_{L^\infty(\Gamma)}^{1/2}  \right)M_2 \notag\\
	&\leq k_{gM}M_2\left( Cl_0+s_1^{1/2}  \right)\|\lambda-\bar \lambda\|_{L^\infty(\Gamma)}^{1/2} \notag\\
		&\leq k_3\|\lambda-\bar \lambda\|_{L^\infty(\Gamma)}^{1/2},
\end{align} 
where $k_3:=k_{gM}M_2\left( Cl_0+s_1^{1/2}  \right)$. 

To estimate  the third term, we first have from \eqref{BoundM} that
\begin{align}
		& |g_{iy}(\cdot, \hat y_\lambda,\lambda)|_{L^\infty(\Gamma)}\leq  k_{gM} M + \sum \limits_{1\leq i\leq m}\|g_{iy}(\cdot, 0, 0)\|_{L^\infty(\Gamma)}:=M_3, \label{Estim3-1}\\
&\|\beta(\bar \lambda )\|_{L^\infty(\Gamma)} \leq M_4,\ \|\beta(\lambda)\|_{L^\infty(\Gamma)}\leq M_4 \quad \text{with } M_4:=k_{\beta M} M + |\beta( 0)|, \label{Estim3-2}\\
&\|\alpha (\bar \lambda)\|_{L^\infty(\Gamma)}\leq M_5,\ \|\alpha (\lambda(x))\|_{L^\infty(\Gamma)}\leq M_5 \quad \text{with } M_5:=k_{\alpha M} M + |\alpha( 0)|, \label{Estim3-3}
\end{align}
for a.e. $x\in \Gamma$ and for all $ i=1,2,...,m$.
Using \eqref{Estim3-1}-\eqref{Estim3-3} and \eqref{HolderCond1},  we get
\begin{align*}
	&\left\|\sum \limits_{i=1}^mg_{iy}[\cdot](\alpha(\bar \lambda)+\beta(\bar \lambda) \bar u)\chi_i-\sum \limits_{i=1}^mg_{iy}(\cdot,\hat y_\lambda,\lambda) (\alpha(\lambda)+\beta(\lambda) \hat u_\lambda)\chi_i\right\|_{L^2(\Gamma)}\\
	&=\left\|\sum \limits_{i=1}^mg_{iy}(\cdot,\hat y_\lambda,\lambda) (\alpha(\lambda)+\beta(\lambda) \hat u_\lambda)\chi_i-\sum \limits_{i=1}^mg_{iy}[\cdot](\alpha(\bar \lambda)+\beta(\bar \lambda) \bar u)\chi_i\right\|_{L^2(\Gamma)}\\
	&=\sum \limits_{i=1}^m \big\|[ g_{iy}(\cdot,\hat y_\lambda,\lambda)\left(\alpha\left(\lambda\right)+\beta\left(\lambda\right) \hat u_\lambda-\alpha\left(\bar \lambda\right)-\beta\left(\bar \lambda\right)\bar  u\right)
	\\
	&\qquad \qquad  \qquad \qquad \qquad +\left(g_{iy}(\cdot,\hat y_\lambda,\lambda)-g_{iy}[\cdot]\right)\left(\alpha\left(\bar \lambda\right)+\beta\left(\bar \lambda\right)\bar  u\right) ]\chi_i \big\|_{L^2( \Gamma)}\\
	&\leq \sum \limits_{i=1}^m \big\|\big(g_{iy}(\cdot,\hat y_\lambda,\lambda)\big(\alpha\left(\lambda\right)-\alpha\left(\bar \lambda\right)\big)\big)\chi_i\big\|_{L^2(\Gamma)} \\
	& \quad +\sum \limits_{i=1}^m \big\|\big(g_{iy}(\cdot,\hat y_\lambda,\lambda)\big(\beta\left(\lambda\right) \hat u_\lambda-\beta\left(\bar \lambda\right)\bar  u\big)\big)\chi_i\big\|_{L^2(\Gamma)}\\
	&\quad + \sum \limits_{i=1}^m \|\big(g_{iy}(\cdot,\hat y_\lambda,\lambda)-g_{iy}[\cdot]\big)\big(\alpha\left(\bar \lambda\right)+\beta\left(\bar \lambda\right)\bar  u\big)\chi_i\|_{L^2(\Gamma)} \\
	&\leq  \sum \limits_{i=1}^m M_3 k_{\alpha M}\|\lambda-\bar \lambda\|_{L^\infty(\Gamma)}\|\chi_i\|_{L^2(\Gamma)} + \sum \limits_{i=1}^m M_3\|(\beta\left( \lambda \right) \hat u_\lambda-\beta\left(\bar \lambda\right)\bar  u)\chi_i \|_{L^2(\Gamma)}
	\\
	&\quad +k_{gM}\left(\|\hat y_\lambda-\bar y\|_{L^\infty(\Omega)} +\|\lambda-\bar \lambda\|_{L^\infty(\Gamma)}\right)\|\alpha\left(\bar \lambda\right)+\beta\left(\bar \lambda\right)\bar  u \|_{L^2(\Gamma)}
\end{align*}
\begin{align*}
\quad\quad& \leq M_3 k_{\alpha M}|\Gamma|\|\lambda-\bar \lambda\|_{L^\infty(\Gamma)} +  \sum \limits_{i=1}^m  M_3\|\left(\beta\left(\lambda\right) \left(\hat u_\lambda-\bar u\right)+ \left(\beta\left(\lambda\right)-\beta\left(\bar \lambda\right)\right)\bar  u\right)\chi_i\|_{L^2(\Gamma)}\\
	&\quad+k_{gM}\left(C l_0\|\lambda-\bar \lambda\|_{L^\infty(\Gamma)}^{1/2} +\|\lambda-\bar \lambda\|_{L^\infty(\Gamma)}\right)(M_5|\Gamma|^{1/2}+M_4\|\bar  u \|_{L^2(\Gamma)})\\
&\leq M_3 k_{\alpha M}|\Gamma|\|\lambda-\bar \lambda\|_{L^\infty(\Gamma)} +  M_3M_4\|\hat u_\lambda-\bar u\|_{L^2(\Gamma)}+2M_3 k_{\beta M}\|\lambda-\bar \lambda\|_{L^\infty(\Gamma)}\|\bar  u\|_{L^2(\Gamma)}\\
&\quad +k_{gM}\left(C l_0\|\lambda-\bar \lambda\|_{L^\infty(\Gamma)}^{1/2} +\|\lambda-\bar \lambda\|_{L^\infty(\Gamma)}\right)(M_5|\Gamma|^{1/2}+M_4\|\bar  u \|_{L^2(\Gamma)})\\
&\leq M_3 k_{\alpha M}|\Gamma|s_1^{1/2}\|\lambda-\bar \lambda\|_{L^\infty(\Gamma)}^{1/2} +  M_3M_4l_0\|\lambda-\bar \lambda\|_{L^\infty(\Gamma)}^{1/2} \\
&\quad +2M_3 k_{\beta M}\|\bar  u\|_{L^2(\Gamma)}s_1^{1/2}\|\lambda-\bar \lambda\|_{L^\infty(\Gamma)}^{1/2}+k_{gM}\left(C l_0 +s_1^{1/2}\right)(M_5|\Gamma|^{1/2}\\
&\quad +M_4\|\bar  u \|_{L^2(\Gamma)})\|\lambda-\bar \lambda\|_{L^\infty(\Gamma)}^{1/2}\\
&\leq k_4\|\lambda-\bar \lambda\|_{L^\infty(\Gamma)}^{1/2},
\end{align*} where 
\begin{align*}
	k_4&:=M_3 k_{\alpha M}|\Gamma|s_1^{1/2} +  M_3M_4l_0+2M_3 k_{\beta M}\|\bar  u\|_{L^2(\Gamma)}s_1^{1/2}v\\
	& \qquad \qquad \qquad
	+k_{gM}\left(C l_0 +s_1^{1/2}\right)\big(M_5|\Gamma|^{1/2}+M_4\|\bar  u \|_{L^2(\Gamma)}\big).
\end{align*} Combining this with \eqref{EstimCase2}-\eqref{BoundTerm2}, we obtain
\begin{align*}%\label{AdjointEstim2}
\|I_\Gamma\|_{L^2(\Gamma)}\leq k_5\|\lambda-\bar \lambda\|_{L^\infty(\Gamma)}^{1/2},
\end{align*} where $k_5= k_2+ k_3+k_4.$  From this and \eqref{EtimAdjointSol}, \eqref{AdjointEstim1} it follows that 
\begin{equation}\label{EtimAdjointSol3}
\|\vartheta-\vartheta_\lambda\|_{W^{1,r}(\Omega)}\leq k_6\|\lambda-\bar \lambda\|_{L^\infty(\Gamma)}^{1/2},
\end{equation} where $k_6= CC_1(k_1 +  k_5)$.

By the assumption $(H5)$, we get $\sum \limits_{i=1}^{m}g_i(x, \hat y_\lambda(x), \lambda(x))\chi_i(x)=\max \limits_{1\leq i\leq m}{ g_i(x, \hat y_\lambda(x), \lambda(x))}$ for a.e. $x\in \Gamma$. 
Analysis similar to that in the proof of Lemma \ref{LemmaReg} 
shows that
\begin{align*}
&\sum \limits_{i=1}^{m}g_i(x, \hat y_\lambda(x), \lambda(x))\chi_i(x) +\hat u_\lambda(x)
\\&\qquad \qquad  \qquad =P_{(- \infty,0]}\left( \frac{1}{\beta(\lambda(x))}(\vartheta_\lambda(x)-\alpha(\lambda(x))) + \sum \limits_{i=1}^{m}g_i(x, \hat y_\lambda(x), \lambda(x))\chi_i(x)\right)
\end{align*}  and
\begin{equation*}
\sum \limits_{i=1}^{m}g_i[x]\chi_i(x) +\bar  u(x)=P_{(- \infty,0]}\left( \frac{1}{2\beta(\bar \lambda(x))}(\vartheta(x)-\alpha(\bar \lambda(x))) + \sum \limits_{i=1}^{m}g_i[x]\chi_i(x)\right)
\end{equation*} for a.e. $x\in \Gamma$. This implies that
\begin{align*}
\hat u_\lambda(x)-\bar  u(x)=&P_{(- \infty,0]}\left( \frac{1}{\beta(\lambda(x))}(\vartheta_\lambda(x)-\alpha(\lambda(x))) + \sum \limits_{i=1}^{m}g_i(x, \hat y_\lambda(x), \lambda(x))\chi_i(x)\right)\\
&-P_{(- \infty,0]}\left( \frac{1}{\beta(\bar \lambda(x))}(\vartheta(x)-\alpha(\bar \lambda(x))) + \sum \limits_{i=1}^{m}g_i[x]\chi_i(x))\right)\\
&+\sum \limits_{i=1}^{m}g_i[x]\chi_i(x))-\sum \limits_{i=1}^{m}g_i(x, \hat y_\lambda(x), \lambda(x))\chi_i(x) \quad \text{a.e. }x\in \Gamma.
\end{align*}Using the non-expansive property of metric projections, we have
\begin{align}\label{EstimFinal}
|\hat u_\lambda(x)-\bar  u(x)|\leq &\Big| \left( \frac{1}{\beta(\lambda(x))}(\vartheta_\lambda(x)-\alpha(\lambda(x))) + \sum \limits_{i=1}^{m}g_i(x, \hat y_\lambda(x), \lambda(x))\chi_i(x)\right)\notag\\
&-\left( \frac{1}{\beta(\bar \lambda(x))}(\vartheta(x)-\alpha(\bar \lambda(x))) + \sum \limits_{i=1}^{m}g_i[x]\chi_i(x)\right)\Big|\notag\\
&+\left|\sum \limits_{i=1}^{m}g_i[x]\chi_i(x)-\sum \limits_{i=1}^{m}g_i(x, \hat y_\lambda(x), \lambda(x))\chi_i(x)\right| \notag\\
\leq &\left|   \frac{1}{\beta(\lambda(x))}\big(\vartheta_\lambda(x)-\alpha(\lambda(x))\big) - \frac{1}{\beta(\bar \lambda(x))}\big(\vartheta(x)-\alpha(\bar \lambda(x))\big) \right|\notag\\
&+2\left|\sum \limits_{i=1}^{m}\big (g_i[x]-g_i(x, \hat y_\lambda(x), \lambda(x))\big)\chi_i(x)\right|. 
\end{align} Let us estimate the terms on the right hand side of \eqref{EstimFinal}. Using $(H3)$, \eqref{StrictNonegative}, \eqref{EtimAdjointSol3} and embedding $W^{1,r}(\Omega)\hookrightarrow C(\bar\Omega)$, we have the following estimation of the first  term:
\begin{align}
	&\left|\frac{1 }{\beta\left(\lambda\left(x\right)\right)}\left(\vartheta_\lambda\left(x\right)\right)-\alpha\left(\lambda\left(x\right)\right)-\frac{1 }{\beta\left(\bar \lambda\left(x\right)\right)}\left(\vartheta\left(x\right)-\alpha\left(\bar \lambda\left(x\right)\right)\right)\right| \notag\\
	&=\left|\frac{1 }{\beta\left(\lambda\left(x\right)\right)}\left(\vartheta_\lambda(x)-\alpha\left(\lambda\left(x\right)\right)- \vartheta\left(x\right)+\alpha(\bar \lambda(x))\right)\right| +\notag\\
	&\qquad+\left|\left(\frac{1 }{\beta\left(\bar \lambda\left(x\right)\right)}-\frac{1 }{\beta\left(\lambda\left(x\right)\right)}\right)\left(\vartheta\left(x\right)-\alpha\left(\bar \lambda\left(x\right)\right)\right)\right|\notag\\
	&\leq \frac{2}{\gamma}\left(|\vartheta_\lambda\left(x\right)-\vartheta\left(x\right)|+ |\alpha\left(\lambda\left(x\right)\right)- \alpha\left(\bar \lambda\left(x\right)\right)|\right)+\notag\\
	&\qquad+\frac{4}{\gamma^2}|\beta\left(\lambda\left(x\right)\right)-\beta\left(\bar \lambda\left(x\right)\right)|\left( |\vartheta\left(x\right)|+|\alpha\left(\bar \lambda\left(x\right)\right)|\right)\notag\\
	& \leq \frac{2}{\gamma}\left(Ck_6\|\lambda-\bar \lambda\|^{1/2}_\infty + k_{\alpha M}|\lambda\left(x\right)-\bar \lambda\left(x\right)| +\frac{2}{\gamma}k_{\beta M}|\lambda\left(x\right)-\bar \lambda\notag\left(x\right)|( M_2+M_5)\right) \notag
\\
&\leq  \frac{2}{\gamma}\left(Ck_6\|\lambda-\bar \lambda\|^{1/2}_\infty + \left(k_{\alpha M} +\frac{2}{\gamma}k_{\beta M} (M_2+M_5)\right)\|\lambda-\bar \lambda\|_{L^\infty(\Gamma)}\right)\qquad \qquad \notag\\
&\leq \frac{2}{\gamma}\left(Ck_6 + \left(k_{\alpha M} +\frac{2}{\gamma}k_{\beta M}( M_2+M_5)\right)s_1^{1/2}\right)\|\lambda-\bar \lambda\|_{L^\infty(\Gamma)}^{1/2} \notag\\
&\leq k_7\|\lambda-\bar \lambda\|_{L^\infty(\Gamma)}^{1/2},
\end{align} where $k_7:=\frac{2}{\gamma}\left(Ck_6 + \big(k_{\alpha M} +\frac{2}{\gamma}k_{\beta M}( M_2+M_5)\big)s_1^{1/2}\right)$. 

To estimate the second term, from $(H2)$ and \eqref{HolderCond1}, we have 
\begin{align}
	2\Big|\sum \limits_{i=1}^{m}\big (g_i[x]-g_i(x, \hat y_\lambda(x), \lambda(x))\big)\chi_i(x)\Big|
	&\leq 2k_{gM}(|\hat y_\lambda(x)-\bar y(x)| +|\lambda(x)-\bar \lambda(x)|) \notag \\
	&\leq 2k_{gM}(\|\hat y_\lambda-\bar y\|_{L^\infty(\bar\Omega)} +\|\lambda-\bar \lambda\|_{L^\infty(\Gamma)}) \notag \\
	&\leq 2k_{gM}(C\|\hat y_\lambda-\bar y\|_{W^{1,r}(\Omega)} +\|\lambda-\bar \lambda\|_{L^\infty(\Gamma)})\notag \\
	&\leq 2k_{gM}(Cl_0\|\lambda-\bar \lambda\|_{L^\infty(\Gamma)}^{1/2} +\|\lambda-\bar \lambda\|_{L^\infty(\Gamma)})\notag \\
	&\leq 2k_{gM}(Cl_0\|\lambda-\bar \lambda\|_{L^\infty(\Gamma)}^{1/2} +s_1^{1/2}\|\lambda-\bar \lambda\|^{1/2}_\infty)\notag \\
	&\leq k_8\|\lambda-\bar \lambda\|_{L^\infty(\Gamma)}^{1/2} \label{EstimSecTerm}
\end{align} for a.e. $x\in \Gamma$, where $k_8= 2k_{gM}\big(Cl_0 +s_1^{1/2}\big)$. \\
Combining \eqref{EstimFinal}-\eqref{EstimSecTerm} we obtain \begin{align*}%\label{estim-key2}
|\hat u_\lambda(x)-\bar u(x)|\leq (k_7+k_8)\|\lambda-\bar \lambda\|_{L^\infty(\Gamma)}^{1/2} \quad \text{a.e.}\ x\in\Gamma.
\end{align*} Thus, setting $k_9:= k_7+k_8$, we have 
\begin{align*}
\|\hat u_\lambda-\bar u\|_{L^\infty(\Gamma)} \leq k_9\|\lambda-\bar \lambda\|_{L^\infty(\Gamma)}^{1/2}.
\end{align*} Combining this with \eqref{HolderCond1} yields 
\begin{align*}
\|\hat y_\lambda-\bar y\|_{W^{1,r}(\Omega)} +\|\hat u_\lambda-\bar u\|_\infty \leq (l_0+k_9)\|\lambda-\bar \lambda\|_{L^\infty(\Gamma)}^{1/2}.
\end{align*} Therefore, the lemma is proved  with the choice $l_1= l_0+ k_9$.    
\end{proof}

The proof of Theorem \ref{Theoremkey} is complete.

\section*{Disclosure statement} 
 The authors report there are no competing interests to declare.
%No potential conflict of interest was reported by the authors.

\section*{Acknowledgments}
%The authors would like to thank the anonymous referees for their comments and suggestions which improve the paper greatly. 
This work is supported by the Vietnam Ministry of Education and Training and Vietnam Institute for
advanced study in Mathematics under Grant B2022-CTT-05.

\end{document}